\documentclass{amsart}

\usepackage{comment}
\usepackage{amsmath}
\usepackage{url}
\usepackage{amssymb}
\usepackage{comment}
\usepackage{subfigure}
\usepackage{accents}
\usepackage{tcolorbox}
\usepackage{tikz,lipsum}
\tcbuselibrary{skins,breakable}
\usetikzlibrary{calc,arrows}
\usepackage{cleveref}
\usepackage{xcolor}
\usepackage{textcomp}
\usepackage{graphicx}

\newcommand{\oprocendsymbol}{\hbox{$\square$}}
\newcommand{\oprocend}{\relax\ifmmode\else\unskip\hfill\fi\oprocendsymbol}

\newtheorem{thm}{Theorem}[section]
\newtheorem{corollary}[thm]{Corollary}
\newtheorem{lemma}[thm]{Lemma}
\newtheorem{proposition}[thm]{Proposition}
\newtheorem{rem}[thm]{Remark}

\newtheorem{assumption}[thm]{Assumption}
\newtheorem*{probf}{Problem statement}

\crefname{assumption}{Assumption}{Assumptions}
\crefname{proposition}{Proposition}{Propositions}
\crefname{corollary}{Corollary}{Corollaries}
\crefname{thm}{Theorem}{Theorems}
\crefname{rem}{Remark}{Remarks}
\crefname{lemma}{Lemma}{Lemmas}
\crefname{figure}{Figure}{Figures}
\crefname{appendix}{Appendix}{}

\newcommand{\subscr}[2]{#1_{\textup{#2}}}
\newcommand{\supscr}[2]{#1^{\textup{#2}}}

\newcommand{\RgeO}{\ensuremath{\mathbb{R}_{ \geq 0}}}
\newcommand{\Rat}[1]{\ensuremath{\mathbb{R}^{#1}}}
\newcommand{\Cdf}{\operatorname{Cdf}}
\newcommand{\Env}{\operatorname{Env}}
\newcommand{\B}{\mathcal{B}}

\newcommand{\CD}{\mathcal{C}\mathcal{D}}

\newcommand{\F}{\mathcal{F}}
\renewcommand{\H}{\mathcal{H}}

\renewcommand{\P}{\mathcal{P}}
\newcommand{\bP}{\mathbb{P}}
\newcommand{\bN}{\mathbb{N}}

\begin{document}

\author[F. Boso]{Francesca Boso*}
\address{Department of Energy Resources Engineering, Stanford University}
\email{fboso@stanford.edu}

\author[D. Boskos]{Dimitris Boskos*}
\address{Department of Mechanical and Aerospace Engineering, UC San Diego}
\email{dboskos@ucsd.edu}

\author[J. Cort\'es]{Jorge Cort\'es}
\address{Department of Mechanical and Aerospace Engineering, UC San Diego}
\email{cortes@ucsd.edu}

\author[S. Mart{\'\i}nez]{Sonia Mart{\'\i}nez}
\address{Department of Mechanical and Aerospace Engineering, UC San Diego}
\email{soniamd@ucsd.edu}

\author[D. M. Tartakovsky]{Daniel M. Tartakovsky}
\address{Department of Energy Resources Engineering, Stanford University}
\email{tartakovsky@stanford.edu}

\begin{abstract}
Ambiguity sets of probability distributions are used to hedge against uncertainty about the true probabilities of random quantities of interest (QoIs). When available, these ambiguity sets are constructed from both data (collected at the initial time and along the boundaries of the physical domain) and concentration-of-measure results on the Wasserstein metric. To propagate the ambiguity sets into the future, we use a physics-dependent equation governing the evolution of cumulative distribution functions (CDF) obtained through the method of distributions.  This study focuses on the latter step by investigating the spatio-temporal evolution of data-driven ambiguity sets and their associated guarantees when the random QoIs they describe obey hyperbolic partial-differential equations with random inputs.  For general nonlinear hyperbolic equations with smooth solutions,  the CDF equation is used to propagate the upper and lower envelopes of pointwise ambiguity bands. For linear dynamics, the CDF equation  allows us to construct an evolution equation for tighter ambiguity balls. We demonstrate that, in both cases, the ambiguity sets are guaranteed to contain the true (unknown) distributions within a prescribed confidence. 
\end{abstract}

\title[Ambiguity Sets for Hyperbolic Conservation Laws]{Dynamics of Data-driven 
Ambiguity Sets for Hyperbolic Conservation Laws with Uncertain Inputs}



\thanks{*Work supported  by  the  DARPA Lagrange program through award 
N66001-18-2-4027. Both first authors contributed equally.}

\maketitle

\section{Introduction}
Hyperbolic conservation laws describe a wide spectrum of engineering applications ranging from multi-phase flows~\cite{buckley1942mechanism} to networked traffic~\cite{lebacque2005first}. The underlying dynamics is described by first-order hyperbolic partial differential equations (PDEs) with non-negligible parametric uncertainty, induced by factors such as limited and/or noisy measurements and random fluctuations of environmental attributes. Decisions based, in whole or in part, on predictions obtained from such models have to account for this uncertainty. The decision maker often has no distributional knowledge of the parametric uncertainties affecting the model and uses data---often noisy and insufficient---to make inferences about these distributions. Robust stochastic programming~\cite{DB-DBB-CC:11} calls for a quantifiable description of sets of probability measures, termed ambiguity sets, that contain the true (yet unknown) distribution with high confidence (e.g.,~\cite{pflug2007ambiguity,esfahani2018data,shapiro2004class}). 
The availability of such sets underpins distributionally robust optimization (DRO) formulations~\cite{DB-DBB-CC:11,shapiro2017distributionally} that are able of hedging against these uncertainties.
%
Ambiguity sets are typically defined either through moment constraints~\cite{ED-YY:10} or statistical metric-like notions such as $\phi$-divergences~\cite{AB-DD-AD-BM-GR:13} and Wasserstein metrics~\cite{esfahani2018data}, which allow the designer to identify distributions that are close to the nominal distribution 
in the prescribed metric.   
Ideally, ambiguity sets should be rich enough to contain the true distribution with high probability; be amenable to tractable reformulations; capture distribution variations relevant to the optimization problem without being overly conservative; and be data-driven. 
Wasserstein ambiguity sets have emerged as an appropriate choice because of two reasons. First, they provide computationally convenient dual reformulations of the associated DRO problems~\cite{esfahani2018data,RG-AJK:16}. Second, they penalize horizontal dislocations of the distributions~\cite{FS:15}, which considerably affect solutions of the stochastic optimization problems. Furthermore, 
data-driven Wasserstein ambiguity sets are accompanied by finite-sample guarantees of containing the true distribution with high confidence~\cite{NF-AG:15,SD-MS-RS:13,JW-FB:19}, resulting in DRO problems with prescribed out-of-sample performance. Our recent work~\cite{DB-JC-SM:19-tac,DB-JC-SM:19} has  explored  how ambiguity sets change under deterministic flow maps generated by ordinary differential equations, and used this information in dynamic DRO formulations. 
For these reasons, Wasserstein DRO formulations are utilized in a wide range of applications including distributed optimization~\cite{AC-JC:20-tac}, machine learning~\cite{JB-YK-KM:19}, traffic control~\cite{DL-DF-SM:19-ecc}, power systems~\cite{YG-KB-ED-ZH-THS:18}, and logistics~\cite{RJ-MR-GX:19}. 

We consider two types of input ambiguity sets. The first is based on Wasserstein balls, whereas the second exploits CDF bands that contain the CDF of the true distribution with high probability. Our focus is on the spatio-temporal evolution of data-driven ambiguity sets (and their associated guarantees) 
when the random quantities they describe obey hyperbolic PDEs with random inputs. 
Many techniques can be used to propagate uncertainty affecting the inputs of a stochastic PDE to its solution. We use the method of distributions (MD)~\cite{tartakovsky2017method}, which yields a deterministic evolution equation for the single-point cumulative distribution function (CDF) of a model output~\cite{boso-2014-cumulative}. This method provides an efficient alternative to numerically demanding Monte Carlo simulations (MCS), which require multiple solutions of the PDE with repeated realizations of the random inputs. It is ideal for hyperbolic problems, for which other techniques (such us stochastic finite elements and stochastic collocation) can be slower than MCS~\cite{boso-2019-hyperbolic}. 
In particular, when uncertainty in initial and boundary conditions is propagated by a hyperbolic deterministic PDE with a smooth solution, MD yields an exact CDF equation~\cite{venturi-2013-exact, boso-2014-cumulative}.
Regardless of the uncertainty propagation technique, data can be used both to characterize the statistical properties of the input distributions and reduce uncertainty by assimilating observations into probabilistic model predictions via Bayesian techniques, e.g.,~\cite{wikle2007bayesian}. 

%
%

The contributions of our study are threefold. First, we use data collected at the initial time and along the boundaries of the physical domain to build ambiguity sets that enjoy rigorous finite-sample guarantees for the input distributions. Specifically, we construct data-driven pointwise ambiguity sets for the \textit{unknown} true distributions of parameterized random inputs, by transferring finite-sample guarantees for their associated Wasserstein distance in the parameter domain. The resulting ambiguity sets account for  empirical information (from the data) without introducing arbitrary hypotheses on the distribution of the random parameters. Second, we design tools to propagate the ambiguity sets throughout space and time. The MD is employed to propagate each ambiguous distribution within the data-driven input ambiguity sets according to a physics-dependent CDF equation. For linear dynamics, we use the CDF equation to construct an evolution equation for the radius of ambiguity balls centered at the empirical distributions in the 1-Wasserstein (a.k.a. Kantorovich) metric. For a wider class of nonlinear hyperbolic equations with smooth solutions, we exploit the CDF equation to propagate the upper and lower envelopes of pointwise ambiguity bands. These are formed through upper and lower envelopes that contain all CDFs up to an assigned 1-Wasserstein distance from the empirical CDF. Third, we use these uncertainty propagation tools to obtain pointwise ambiguity sets across all locations of the space-time domain that contain their true distributions with prescribed probability. Our method can handle both types of input ambiguity sets (based on either Wasserstein balls or CDF bands), while maintaining their confidence guarantees upon propagation. This allows the decision maker to  map their physics-driven stretching/shrinking under the PDE dynamics.

\section{Preliminaries}\label{sec:prelims}
Let $\|\cdot\|$ and $\|\cdot\|_{\infty}$ denote the Euclidean and infinity norm in $\Rat{n}$, respectively. The diameter of a set $S\subset\Rat{n}$ is defined as ${\rm diam}(S):=\sup\{\|x-y\|_{\infty}\,|\,x,y\in S\}$. The Heaviside function $\H:\Rat{}\to\Rat{}$ is $\H(x)=0$ for $x<0$ and $\H(x)=1$ for $x\ge 0$. We denote by $\B(\Rat{d})$ the Borel $\sigma$-algebra on $\Rat{d}$, and by $\P(\Rat{d})$ the space of probability measures on $(\Rat{d},\B(\Rat{d}))$. 
For $\mu\in\P(\Rat{d})$, its support is the closed 	set ${\rm supp}(\mu):=\{x\in\Rat{d}\,|\,\mu(U)>0\;\text{for each neighborhood}\;U\;{\rm of}\;x\}$ or, equivalently, the smallest closed set with measure one. We denote by $\Cdf[P]$ the cumulative distribution function associated with the probability measure $P$ on $\Rat{}$ and by $\CD(I)$ the set of all CDFs of scalar random variables whose induced probability measures are supported on the interval  $I\subset\mathbb R$. Given $p\ge 1$, $\P_p(\Rat{d}):=\{\mu\in\P(\Rat{d})\,|\,\int_{\Rat{d}}\|x\|^pd\mu<\infty\}$ is the set of probability measures in $\P(\Rat{d})$ with finite $p$-th moment. The  Wasserstein distance of $\mu,\nu\in\P_p(\Rat{d})$ is 
\begin{align*}
W_p(\mu,\nu):=\Big(\inf_{\pi\in\mathcal M(\mu,\nu)} \Big\{\int_{\Rat{d}\times\Rat{d}}\|x-y\|^p \pi(dx,dy)\Big\}\Big)^{1/p},
\end{align*}
where $\mathcal M(\mu,\nu)$ is the set of all probability measures on $\Rat{d}\times\Rat{d}$ with marginals $\mu$ and $\nu$, respectively, also termed couplings. For scalar random variables, the Wasserstein distance $W_p$ between two distributions $\mu$ and $\nu$ with CDFs $F$ and $G$ is, cf.~\cite{CV:03},
$W_p(\mu,\nu)=\big(\int_0^1|F^{-1}(t)-G^{-1}(t)|^p dt\big)^{1/p}$,
where $F^{-1}$ denotes the generalized inverse of $F$, $F^{-1}(y)=\inf\{t\in\Rat{}\,|\,F(t)>y\}$. For $p=1$, one can use the representation
\begin{align}\label{eq:W1def}
W_1(\mu,\nu)=\int_\Rat{}|F(s)-G(s)|ds.
\end{align}
Given two measurable spaces $(\Omega,\F)$ and $(\Omega',\F')$, and a measurable function $\Psi$ from $(\Omega,\F)$ to $(\Omega',\F')$, the push-forward map $\Psi_{\#}$ assigns to each measure $\mu$ in $(\Omega,\F)$ a new measure $\nu$ in $(\Omega',\F')$ defined by $\nu:=\Psi_{\#}\mu$ iff $\nu(B)=\mu(\Psi^{-1}(B))$ for all $B\in\F'$. The map $\Psi_{\#}$ is linear and satisfies $\Psi_{\#}\delta_{\omega}=\delta_{\Psi(\omega)}$ with $\delta_{\omega}$ the Dirac mass at $\omega\in\Omega$.

\section{Problem formulation} \label{sec:pf}

We consider a hyperbolic model for~$u(\mathbf x,t)$,
\begin{align}\label{eq:tran}
& \frac{\partial u}{\partial t} + \nabla \cdot \left( \mathbf q(u; \boldsymbol \theta_q) \right) = 
r(u;  \boldsymbol \theta_r), \quad \mathbf x \in \Omega, \quad t>0 
\end{align}
subject to initial and boundary conditions
\begin{align}\label{eq:ICsBCs}
& u(\mathbf x,t=0) = u_0(\mathbf x), \quad \mathbf x \in \Omega \notag \\
& u(\mathbf x,t) = u_b(\mathbf x,t), \quad \mathbf x \in \Gamma, \quad t>0,
\end{align}
restricting ourselves to problems with smooth solutions. Equation~\eqref{eq:tran}, with the given flux $\mathbf q(u; \boldsymbol \theta_q)$ and source term $r(u;\boldsymbol \theta_r)$, is defined on a $d$-dimensional semi-infinite spatial domain $\Omega \subset \mathbb R^d$, and by the parameters $\boldsymbol \theta_q$ and $\boldsymbol \theta_r$, that can be spatially and/or temporally varying. The boundary function $u_b(\mathbf x,t)$ is prescribed at the upstream boundary $\Gamma$. For the sake of brevity, we do not consider different types of boundary conditions, although the procedure can be adjusted accordingly. Randomness in the initial and/or boundary conditions, $u_0(\mathbf x)$ and $u_b(\mathbf x, t)$, renders~\eqref{eq:tran} stochastic. We make the following hypotheses.

\begin{assumption}[Deterministic dynamics]
\label{assumption:exact:dynamics}
We assume all parameters in \eqref{eq:tran} (i.e., all physical parameters specifying the flux $\mathbf q$, $\boldsymbol \theta_q$, and the source term $r$, $\boldsymbol \theta_r$) are deterministic, and the flux $\mathbf q$ is divergence-free once evaluated for a specific value of $u(\mathbf x,t) = U$, $\nabla \cdot \mathbf q(U; \boldsymbol \theta_q) = 0$. 
\end{assumption}

\begin{assumption}[Existence and uniqueness of local solutions within a time horizon]
\label{assumption:unif:existence}
There exists $T\in (0,\infty]$ such that for each initial and boundary condition 
from their probability space, the solution $u(\mathbf x,t)$ of~\eqref{eq:tran} is smooth  
and defined on $\Omega\times[0,T)$.
\end{assumption}

Regarding \cref{assumption:unif:existence},
we refer to~\cite{racke1992lectures} for a theoretical treatment of local existence theorems.
In the \textit{absence of direct access to the distribution of the initial and boundary conditions}, we analyze their samples from independent realizations of~\eqref{eq:ICsBCs}. Specifically, we measure the initial condition $u_0$ for all $\mathbf x\in\Omega$ and get continuous measurements of $u_b$ at each boundary point for all times (for instance, in a traffic flow scenario with $\Omega$ representing a long highway segment, a traffic helicopter might pass above the area at the same time each morning and take a photo from the segment that provides the initial condition for the traffic density $u$, whereas $u$ at the segment boundary is continuously measured by a single-loop detector. Assumptions~\ref{assumption:exact:dynamics} and~\ref{assumption:unif:existence} require traffic conditions far from congestion, with  deterministic parameters describing the flow, specifically maximum velocity and maximum traffic density). 
We are interested in exploiting
the samples to construct ambiguity sets that contain the temporally- and spatially-variable one-point probability distributions of $u_0(t)$ and $u_b(\mathbf x,t)$ with high confidence. We consider initial and boundary conditions that are specified by a finite number of random parameters.  

\begin{assumption}[Input parameterization]
\label{assumption:parameterization}
The initial and boundary conditions are parameterized by $\mathbf a:=(a_1,\ldots,a_n)$ from a compact subset of $\Rat n$, i.e., $u_0(\mathbf x)\equiv u_0(\mathbf x;\mathbf a)$ and $u_b(\mathbf x,t)\equiv u_b(\mathbf x,t;\mathbf a)$. The parameterizations are  globally Lipschitz with respect to $\mathbf a$ for each initial position $\mathbf x$ and boundary pair $(\mathbf x,t)$. Specifically, 
\begin{subequations}
	\begin{align}
	|u_0(\mathbf x;\mathbf a)-u_0(\mathbf x;\mathbf a')| & \le 
	L_0(\mathbf x)\|\mathbf a-\mathbf a'\|\quad \forall \mathbf 
	x\in\Omega,\quad \mathbf a,\mathbf a'\in\Rat{n}, \label{Lip:constant:init} \\ 
	|u_b(\mathbf x,t;\mathbf a)-u_b(\mathbf x,t;\mathbf a')| & \le 
	L_b(\mathbf x,t)\|\mathbf a-\mathbf a'\|\quad \forall \mathbf 
	x\in\Gamma,\quad t\ge 0, \quad \mathbf a,\mathbf a'\in\Rat n, \label{Lip:constant:bnd}
	\end{align}
\end{subequations}
for some continuous functions $L_0:\Omega\to\RgeO$ and $L_b:\Gamma\times\RgeO\to\RgeO$.
\end{assumption}

We denote by  $\supscr{P}{true}_{\mathbf a}$ the distribution of the parameters in $\Rat n$, by $\supscr{P}{true}_{u_0(\mathbf x)}$ the induced distribution of
$u_0(\mathbf x;\mathbf a)$ at the spatial point $\mathbf x$, and by 
$\supscr{P}{true}_{u_b(\mathbf x,t)}$ the distribution of $u_b(\mathbf 
x,t;\mathbf a)$ at each boundary point $\mathbf x$ and time $t\ge 0$. We use the 
superscript `${\rm true}$' to emphasize that we refer to the 
corresponding true distributions, that are unknown. We denote by 
$\supscr{F}{true}_{u_0(\mathbf x)}\equiv\Cdf\big[\supscr{P}{true}_{u_0(\mathbf 
x)}\big]$ and $\supscr{F}{true}_{u_b(\mathbf x,t)}\equiv\Cdf \big[\supscr{P}{true}_{u_b(\mathbf x,t)}\big]$ their associated CDFs and make the following hypothesis for data assimilation. 

\begin{assumption}[Input samples] 
\label{assumption:sampling}
We have access to $N$ independent pairs of initial and boundary condition samples, $(u_0^1,u_b^1),\ldots,(u_0^N,u_b^N)$, generated by corresponding 
independent realizations $\mathbf a^1,\ldots,\mathbf a^N$ of the parameters 
in Assumption~\ref{assumption:parameterization}.
\end{assumption}

Under these hypotheses, we seek to derive \textit{pointwise characterizations} 
of ambiguity sets for the CDF of $u$ at each location $(\mathbf x,t)$ in space and time, starting with their characterization for the initial and boundary data. We are interested in defining the ambiguity sets in terms of plausible CDFs at each $(\mathbf x,t)$, and exploiting the known dynamics~\eqref{eq:tran} to propagate the one-point CDFs of $u(\mathbf x,t)$ in space and time.   

\begin{probf}
Given $\beta$, we seek to determine sets $\P_{\mathbf 
x}^0$, $\mathbf x\in\Omega$ and $\P_{\mathbf x,t}^b$, $(\mathbf 
x,t)\in\Gamma\times\RgeO$ of CDFs that contain the corresponding true CDFs $\supscr{F}{true}_{u_0(\mathbf x)}$ and $\supscr{F}{true}_{u_b(\mathbf x,t)}$ for the initial and boundary conditions, respectively, with confidence $1-\beta$,
\begin{align*}
\bP(\{\supscr{F}{true}_{u_0(\mathbf x)}\in \P_{\mathbf x}^0\;\forall \mathbf 
x\in\Omega\}\cap\{\supscr{F}{true}_{u_b(\mathbf x,t)}\in \P_{\mathbf 
x,t}^b\;\forall 
(\mathbf x,t)\in\Gamma\times\RgeO\})\ge 1-\beta.
\end{align*} 
We further seek to leverage the PDE dynamics to propagate the  
ambiguity sets of the initial and boundary data and obtain a pointwise 
characterization of ambiguity sets $\P_{\mathbf x,t}$ containing 
the CDF of $u(\mathbf x,t)$ at each $\mathbf x\in \Omega$ and $t\in[0,T)$ with 
confidence $1-\beta$,   
\begin{align*}
\bP(\supscr{F}{true}_{u(\mathbf x,t)}\in \P_{\mathbf x,t}\;\forall (\mathbf 
x,t)\in\Omega\times[0,T))\ge 1-\beta.
\end{align*}    
\end{probf}
Section~\ref{sec:init:guarantees} exploits the compactly supported parameterization of the initial and boundary data to build ambiguity sets which enjoy rigorous finite-sample guarantees. Section~\ref{sec:CDF:diff} derives a deterministic PDE for the CDF of $u(\mathbf x,t)$, which enables the investigation of how the difference between CDFs (and, by integration, their Wasserstein distance) evolves in space and time. Section~\ref{sec:ambiguity:evolution} characterizes how the input ambiguity sets propagate in space and time under the same confidence guarantees. 


\section{Data-driven ambiguity sets for inputs}\label{sec:init:guarantees}

%
Using \cref{assumption:parameterization,assumption:sampling}, at each $\mathbf x\in\Omega$ and boundary pair $(\mathbf 
x,t)\in\Gamma\times\RgeO$, we define empirical distributions 
\begin{align*}
\widehat P_{u_0(\mathbf x)}^N \equiv 
\widehat P_{u_0(\mathbf x)}^N(\mathbf a^1,\ldots,\mathbf a^N)
& := \frac{1}{N}\sum_{i=1}^N\delta_{u_0^i(\mathbf x)}\equiv 
\frac{1}{N}\sum_{i=1}^N\delta_{u_0(\mathbf x;\mathbf a^i)},\\ 
\widehat P_{u_b(\mathbf x,t)}^N \equiv
\widehat P_{u_b(\mathbf x,t)}^N (\mathbf a^1,\ldots,\mathbf a^N)
& := \frac{1}{N}\sum_{i=1}^N\delta_{u_b^i(\mathbf x,t)}
\equiv\frac{1}{N}\sum_{i=1}^N\delta_{u_b(\mathbf x,t;\mathbf a^i)}, 
\end{align*} 
with associated CDFs $\widehat F_{u_0(\mathbf x)}^N:=\Cdf\big[\widehat P_{u_0(\mathbf x)}^N\big]$ and $\widehat F_{u_b(\mathbf x,t)}^N:=\Cdf\big[\widehat P_{u_b(\mathbf x,t)}^N\big]$. We employ these empirical distributions to build \textit{pointwise ambiguity sets} based on concent\-ration-of-measure results for the 1-Wasserstein distance.
Specifically, 
we exploit compactness of the initial and boundary data parameterization together with the following confidence guarantees about the Wasserstein distance between the empirical and true distribution of compactly supported random variables (see~\cite{DB-JC-SM:19}).  

\begin{lemma}[Ambiguity radius]
\label{lemma:ambiguity:radius}
Let $(X_i)_{i\in\bN}$ be a sequence of i.i.d. $\Rat{n}$-valued random 
variables that have a compactly supported distribution $\mu$ and let   
$\rho:={\rm diam}({\rm supp}(\mu))/2$. Then, for $p\ge 1$, $N\ge 1$, and 
$\epsilon>0$, $\bP(W_p(\widehat{\mu}^N,\mu)\le\epsilon_N(\beta,\rho))\ge 
1-\beta$, where
\begin{align}\label{epsN:dfn}
\epsilon_{N}(\beta,\rho):=
\begin{cases}
\left(\frac{\ln(C\beta^{-1})}{c}\right)^{\frac{1}{2p}}\frac{\rho}{N^{\frac{1}{2p}}},
& {\rm if}\; p>n/2, \\
h^{-1}\left(\frac{\ln(C\beta^{-1})}{cN}\right)^{\frac{1}{p}}\rho, & {\rm if}\; 
p=n/2, \\
\left(\frac{\ln(C\beta^{-1})}{c}\right)^{\frac{1}{n}}\frac{\rho}{N^{\frac{1}{n}}},
& {\rm if}\; p<n/2,
\end{cases}
\end{align}
$\widehat{\mu}^N:=\frac{1}{N}\sum_{i=1}^N\delta_{X_i}$, the constants 
$C$ and $c$ depend only on $p$, $n$, and $h^{-1}$ is the inverse of 
$h(x)= x^2 / [\ln(2+1/x)]^2$, $x>0$. 
\end{lemma}
%
This result quantifies the radius $\epsilon_{N}(\beta,\rho)$ of an ambiguity ball that contains the true distribution with high probability. The radius decreases with the number of samples and can be tuned by the confidence level $1-\beta$, allowing the decision maker to choose the desired level of conservativeness.
The explicit determination of $c$ and $C$ in \eqref{epsN:dfn} through the analysis in~\cite{NF-AG:15} for the whole spectrum of data dimensions $n$ and Wasserstein exponents $p$ can become cumbersome. Nevertheless, \eqref{epsN:dfn} provides explicit ambiguity radius ratios for any pair of sample sizes once a confidence level is fixed. Recall that, according to \cref{assumption:parameterization}, the mapping of the parameters to the initial and boundary data is globally Lipschitz. The following result, whose proof is given in \cref{sec:app:to:sec:init:guarantees}, is useful to quantify the Wasserstein distance between the true and empirical distribution at each input location.

\begin{lemma}[Wasserstein distance under Lipschitz maps] 
\label{lemma:Lipschitz:map}
If  $T:\Rat{n}\to\Rat{m}$ is Lipschitz with constant $L>0$, namely, $\|T(x)-T(y)\|\le L\|x-y\|$, then for any pair of distributions $\mu$, $\nu$ on $\Rat{n}$ it holds that $W_p(\mu,\nu)\le LW_p(T_{\#}\mu,T_{\#}\nu)$. 
\end{lemma} 
Using \cref{lemma:ambiguity:radius,lemma:Lipschitz:map} together with the finite-sample guarantees in the parameter domain, we next obtain a characterization of initial and boundary value ambiguity sets through pointwise Wasserstein balls. To express the ambiguity sets in terms of CDFs, we will interchangeably denote by $W_p(F_{X_1},F_{X_2})\equiv W_p(P_{X_1},P_{X_2})$ the Wasserstein distance between any two scalar random variables $X_1$, $X_2$ with distributions $P_{X_1}$, $P_{X_2}$ and associated CDFs $F_{X_1}=\Cdf[P_{X_1}]$, $F_{X_2}=\Cdf[P_{X_2}]$.

\begin{proposition}[Input ambiguity sets]
\label{prop:init:bnd:ambiguity:sets}
Assume that $N$ pairs of input samples are collected according to \cref{assumption:sampling} and let 
\begin{align} \label{rho:a}
\rho_{\mathbf a}:={\rm diam}({\rm supp}(\supscr{P}{true}_{\mathbf a}))/2
\end{align}
and $\bar{\mathbf a} \in\Rat{n}$ such that $\|\mathbf a-\bar{\mathbf a}\|_{\infty}\le\rho_{\mathbf a}$ for all $\mathbf a\in {\rm supp}(\supscr{P}{true}_{\mathbf a})$. Given a confidence level $1-\beta$, define the ambiguity sets 
\begin{align*}
\P_{\mathbf x}^0 & := \big\{F\in \CD([\alpha_0(\mathbf x),\gamma_0(\mathbf 
x)])\,|\,W_1(\widehat F_{u_0(\mathbf x)}^N,F)\le L_0(\mathbf 
x)\epsilon_N(\beta,\rho_{\mathbf a})\big\} \\
\P_{\mathbf x,t}^b & := \big\{F\in \CD([\alpha_b(\mathbf x,t),\gamma_b(\mathbf x,t)])\,|\,W_1(\widehat F_{u_b(\mathbf x,t)}^N,F)\le L_b(\mathbf x,t) \epsilon_N(\beta,\rho_{\mathbf a})\big\}, 
\end{align*}
for $\mathbf x\in\Omega$ and $\mathbf x\in\Gamma$, $t\ge 0$, respectively, 
where 
\begin{subequations} 
\begin{align}
[\alpha_0(\mathbf x),\gamma_0(\mathbf x)] & := [u_0(\mathbf x;\bar{\mathbf a})-\sqrt{n}L_0(\mathbf 
x)\rho_{\mathbf a},u_0(\mathbf x;\bar{\mathbf a})+\sqrt{n}L_0(\mathbf x)\rho_{\mathbf a}] \label{interv:alpha:gamma:0} \\
[\alpha_b(\mathbf x,t),\gamma_b(\mathbf x,t)] & := [u_b(\mathbf x,t;\bar{\mathbf a})-\sqrt{n}L_b(\mathbf 
x,t)\rho_{\mathbf a},u_b(\mathbf x,t;\bar{\mathbf a})+\sqrt{n}L_b(\mathbf x,t)\rho_{\mathbf a}], \label{interv:alpha:gamma:b} 
\end{align}
\end{subequations}
and $L_0(\mathbf x)$,  $L_b(\mathbf x,t)$, and $\epsilon_N(\beta,\rho_{\mathbf a})$  are given by~\eqref{Lip:constant:init}, \eqref{Lip:constant:bnd}, and 
\eqref{epsN:dfn}. Then, 
\begin{align}
\bP(\{\supscr{F}{true}_{u_0(\mathbf x)}\in \P_{\mathbf x}^0\;\forall 
\mathbf x\in\Omega\}\cap\{\supscr{F}{true}_{u_b(\mathbf x,t)}\in \P_{\mathbf 
x,t}^b\;\forall (\mathbf x,t)\in\Gamma\times\RgeO\}) \ge 1-\beta. 
\label{guarantee:init:bnd} 
\end{align}
\end{proposition} 

\begin{proof}
For the selected confidence $1-\beta$, we get from 
Lemma~\ref{lemma:ambiguity:radius} with $p=1$ that    
\begin{align} \label{guarantee:parameters}
\bP(W_1(\widehat P_{\mathbf a}^N, \supscr{P}{true}_{\mathbf 
a})\le\epsilon_N(\beta,\rho_{\mathbf a}))\ge 1-\beta.
\end{align}  
%
%
Denoting by $u_0[\mathbf x]$ the mapping $\mathbf a\mapsto 
u_0[\mathbf x](\mathbf a):= u_0(\mathbf x;\mathbf a)$, it follows from elementary properties of the pushforward map given in \cref{sec:prelims} that $\widehat P_{u_0(\mathbf x)}^N = u_0[\mathbf x]_\#\widehat P_{\mathbf a}^N$ and $\supscr{P}{true}_{u_0(\mathbf x)} = u_0[\mathbf x]_\#\supscr{P}{true}_{\mathbf a}$, where $\widehat P_{\mathbf a}^N := \frac{1}{N}\sum_{i=1}^N\delta_{\mathbf a^i}$. Thus, we obtain from 
the Lipschitz hypothesis \eqref{Lip:constant:init} and 
Lemma~\ref{lemma:Lipschitz:map} that 
\begin{align*}
W_1(\widehat P_{u_0(\mathbf x)}^N, \supscr{P}{true}_{u_0(\mathbf x)})\le 
L_0(\mathbf x)W_1(\widehat P_{\mathbf a}^N, \supscr{P}{true}_{\mathbf a}), \quad 
\forall \mathbf x\in\Omega.
\end{align*}
Since $\supscr{P}{true}_{u_0(\mathbf x)} = u_0[\mathbf x]_\#\supscr{P}{true}_{\mathbf a}$, we get from \eqref{Lip:constant:init}, \eqref{interv:alpha:gamma:0}, and the selection of $\bar{\mathbf a}$ that $\supscr{P}{true}_{u_0(\mathbf x)}$ is supported on $[\alpha_0(\mathbf x),\gamma_0(\mathbf x)]$, and hence, that $\supscr{F}{true}_{u_0(\mathbf x)}\in\CD([\alpha_0(\mathbf x),\gamma_0(\mathbf x)])$ for all $\mathbf x\in\Omega$. Analogously, we have that 
\begin{align*}
W_1(\widehat P_{u_b(\mathbf x,t)}^N, \supscr{P}{true}_{u_b(\mathbf x,t)})\le 
L_b(\mathbf x,t)W_1(\widehat P_{\mathbf a}^N, \supscr{P}{true}_{\mathbf 
a})
\end{align*}
and $\supscr{F}{true}_{u_b(\mathbf x,t)}\in\CD([\alpha_b(\mathbf x,t),\gamma_b(\mathbf x,t)])$  for all $(\mathbf x,t)\in\Gamma\times\RgeO$. Consequently 
\begin{align*}
& \{W_1(\widehat P_{\mathbf a}^N, \supscr{P}{true}_{\mathbf 
a})\le\epsilon_N(\beta,\rho_{\mathbf a})\}\subset   
\{W_1(\widehat P_{u_0(\mathbf x)}^N, \supscr{P}{true}_{u_0(\mathbf x)})\le 
L_0(\mathbf x)\epsilon_N(\beta,\rho_{\mathbf a})\;\forall \mathbf x\in\Omega\} \\
& \hspace{8em}\cap \{W_1(\widehat P_{u_b(\mathbf x,t)}^N, 
\supscr{P}{true}_{u_b(\mathbf 
x,t)})\le L_b(\mathbf x,t)\epsilon_N(\beta,\rho_{\mathbf a})\;\forall (\mathbf 
x,t)\in\Gamma\times\RgeO\}.
\end{align*}
Thus, since each $\supscr{F}{true}_{u_0(\mathbf x)}\in\CD([\alpha_0(\mathbf x),\gamma_0(\mathbf x)])$ and $\supscr{F}{true}_{u_b(\mathbf x,t)}\in\CD([\alpha_b(\mathbf x,t),\gamma_b(\mathbf x,t)])$, we deduce \eqref{guarantee:init:bnd} from the definitions of the ambiguity sets.
\end{proof}

We next consider an alternative characterization of the ambiguity sets, which enables the exploitation of a propagation tool applicable to a wider class of PDE dynamics, yet at the cost of increased conservativeness.
These ambiguity sets are built using pointwise confidence bands (thereinafter termed \emph{ambiguity bands}), enclosed between upper and lower CDF envelopes that contain the true CDF at each spatio-temporal location with prescribed probability. We rely on the next result, whose proof is given in \cref{sec:app:to:sec:init:guarantees}, providing upper and lower CDF envelopes for any CDF $F$ and distance $\rho$, cf. \cref{fig:CDF:envelope}, so that the CDF of any distribution with 1-Wasserstein distance at most $\rho$ from $F$ is pointwise between these envelopes. 

\begin{figure}[tbh]
\centering
\includegraphics[width=.75\textwidth]{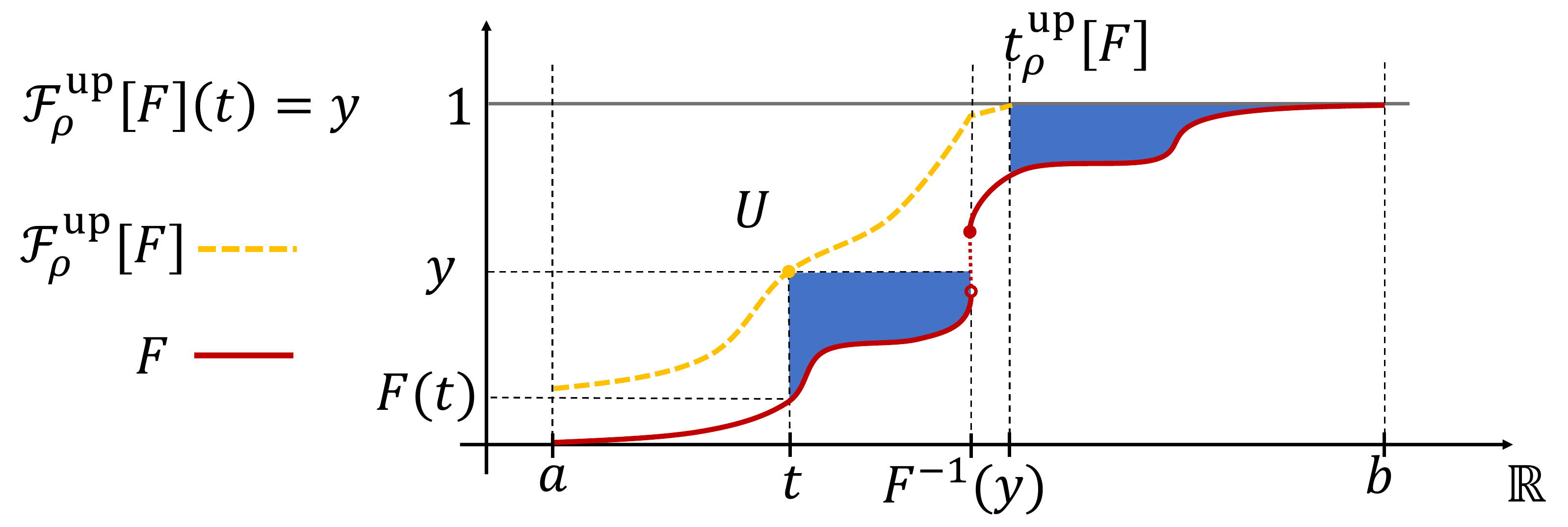}
\caption{Illustration of the upper CDF envelope $\supscr{\F}{up}_\rho[F]$ (in yellow) of $F$ (in red). For each point $(t,y)$ in the graph of $\supscr{\F}{up}_\rho[F]$, the blue area enclosed among the lines parallel to the axes that originate from $(t,y)$ and $F$ is equal to $\rho$.}
\label{fig:CDF:envelope}
\end{figure}


\begin{lemma}[Upper and lower CDF envelopes]\label{lemma:CDF:envelope} 
Let $F\in\CD([a,b])$, define
\begin{align*}
\supscr{t}{up}_\rho[F]\equiv\supscr{t}{up}_{\rho,[a,b]}[F] & 
:=\sup \Big\{\tau\in[a,b]\,\Big|\, \int_\tau^b(1-F(t))dt\ge\rho\Big\} \\
\supscr{t}{low}_\rho[F]\equiv\supscr{t}{low}_{\rho,[a,b]}[F] & 
:=\inf\Big\{\tau\in[a,b]\,\Big|\, \int_a^\tau F(t)dt\ge\rho\Big\}
\end{align*}
for any $0<\rho\le\min\{\int_a^bF(t)dt,\int_a^b(1-F(t))dt\}$, and the corresponding upper and lower CDF envelopes 
$\supscr{\F}{up}_\rho[F]\equiv\supscr{\F}{up}_{\rho,[a,b]}[F]$ and $\supscr{\F}{low}_\rho[F]\equiv\supscr{\F}{low}_{\rho,[a,b]}[F]$ 
\begin{align*}
\supscr{\F}{up}_\rho[F](t) & 
:=\begin{cases}
0, & {\rm if}\;t\in(-\infty,a) \\
\sup\Big\{z\in[F(t),1]\,\big|\, \int_{F(t)}^z(F^{-1}(y)-t)dy\le\rho\Big\}, & 
{\rm if}\;t\in[a,\supscr{t}{up}_\rho[F]) \\
1, & {\rm if}\;t\in[\supscr{t}{up}_\rho[F],\infty),
\end{cases} \\
\supscr{\F}{low}_\rho[F](t) & 
:=\begin{cases}
0, & {\rm if}\;t\in(-\infty,\supscr{t}{low}_\rho[F]) \\
\inf\Big\{z\in[0,F(t)]\,\big|\, \int_z^{F(t)}(t-F^{-1}(y))dy\le\rho\Big\}, & 
{\rm if}\;t\in[\supscr{t}{low}_\rho[F],b) \\
1, & {\rm if}\;t\in[b,\infty). 
\end{cases}
\end{align*}
Then, both  $\supscr{\F}{up}_\rho[F]$ and $\supscr{\F}{low}_\rho[F]$ are continuous CDFs in $\CD([a,b])$ and for any $F'\in\CD([a,b])$ with $W_1(F,F')\le\rho$, it holds that
\begin{align} \label{CDF:sandwich}
\supscr{\F}{low}_\rho[F](t)\le F'(t)\le 
\supscr{\F}{up}_\rho[F](t) , \quad\forall t\in\Rat{}.
\end{align}   
\end{lemma} 

We rely on \cref{lemma:CDF:envelope} to obtain in the next result ambiguity bands for the inputs that share the confidence guarantees with the ambiguity sets of \cref{prop:init:bnd:ambiguity:sets}.

\begin{corollary}[Input ambiguity bands] \label{corollary:CDF:ambig:sets}
Assume $N$ pairs of input samples are collected according to \cref{assumption:sampling} and let $\rho_{\mathbf a}$ and $\bar{\mathbf a}$ as in the statement of \cref{prop:init:bnd:ambiguity:sets}. Given a confidence level $1-\beta$, define the ambiguity sets 
\begin{subequations} 
\begin{align*}
\P_{\mathbf x}^{0,\Env} & := \big\{F\in \CD(\Rat{})\,|\, 
\supscr{\F}{low}_{\rho_0(\mathbf x),[\alpha_0(\mathbf x),\gamma_0(\mathbf 
x)]}[\widehat F_{u_0(\mathbf x)}^N](U)\le F(U) \\
& \hspace{13em} \le \supscr{\F}{up}_{\rho_0(\mathbf x),[\alpha_0(\mathbf 
x),\gamma_0(\mathbf 
x)]}[\widehat F_{u_0(\mathbf x)}^N](U)\;\forall U\in\Rat{}\big\}, \\
\P_{\mathbf x,t}^{b,\Env} & := \big\{F\in \CD(\Rat{})\,|\, 
\supscr{\F}{low}_{\rho_b(\mathbf x,t),[\alpha_b(\mathbf x,t),\gamma_b(\mathbf 
x,t)]}[\widehat F_{u_b(\mathbf x,t)}^N](U)\le F(U) \\
& \hspace{13em} \le \supscr{\F}{up}_{\rho_b(\mathbf x,t),[\alpha_b(\mathbf 
x,t),\gamma_b(\mathbf x,t)]}[\widehat F_{u_b(\mathbf x,t)}^N](U)\;\forall U\in\Rat{}\big\},   
\end{align*}
\end{subequations}
for $\mathbf x\in\Omega$ and $(\mathbf x,t)\in\Gamma\times\RgeO$, respectively, where 
\begin{subequations}
\begin{align}
\rho_0(\mathbf x) & :=L_0(\mathbf x)\epsilon_N(\beta,\rho_{\mathbf a}) \label{rho:zero} \\
\rho_b(\mathbf x,t) & :=  L_b(\mathbf x,t)\epsilon_N(\beta,\rho_{\mathbf a}), \label{rho:beta}
\end{align}
\end{subequations}
and $[\alpha_0(\mathbf x),\gamma_0(\mathbf x)]$, $[\alpha_b(\mathbf x,t),\gamma_b(\mathbf x,t)]$, $\epsilon_N(\beta,\rho_{\mathbf a})$ given by \eqref{interv:alpha:gamma:0}, \eqref{interv:alpha:gamma:b}, and \eqref{epsN:dfn}.~Then 
\begin{align} \label{guarantee:init:bnd:env} 
\bP(\{\supscr{F}{true}_{u_0(\mathbf x)}\in \P_{\mathbf x}^{0,\Env}\;\forall 
\mathbf x\in\Omega\}\cap\{\supscr{F}{true}_{u_b(\mathbf x,t)}\in \P_{\mathbf 
x,t}^{b,\Env}\;\forall (\mathbf x,t)\in\Gamma\times\RgeO\}) \ge 1-\beta. 
\end{align}
\end{corollary} 
\begin{proof}
By \eqref{guarantee:init:bnd} and \eqref{guarantee:init:bnd:env}, it suffices to show that $\P_{\mathbf x}^0\subset \P_{\mathbf x}^{0,\Env}$ and $\P_{\mathbf x,t}^b\subset \P_{\mathbf x,t}^{b,\Env}$ for all $\mathbf x\in\Omega$ and $(\mathbf x,t)\in\Omega\times\RgeO$, respectively, with  $\P_{\mathbf x}^0$ and  $\P_{\mathbf x,t}^b$ given in \cref{prop:init:bnd:ambiguity:sets}. Let $\mathbf x\in\Omega$ and $F\in\P_{\mathbf x}^0$. Then, we get from the definition of $\P_{\mathbf x}^0$ and  \eqref{rho:zero} that $F\in\CD([\alpha_0(\mathbf x),\gamma_0(\mathbf x)])$ and 
$W_1(\widehat F_{u_0}^N,F)\le L_0(\mathbf x)\epsilon_N(\beta,\rho_{\mathbf a})=\rho_0(\mathbf x)$.
Thus, since $F\in\CD([\alpha_0(\mathbf x),\gamma_0(\mathbf x)])$, we can invoke Lemma~\ref{lemma:CDF:envelope} and deduce from \eqref{CDF:sandwich} that  $F\in\P_{\mathbf x}^{0,\Env}$. Analogously, $\P_{\mathbf x,t}^b\subset \P_{\mathbf x,t}^{b,\Env}$ for all $(\mathbf x,t)\in\Omega\times\RgeO$. 
\end{proof}

\begin{rem}[Confidence bands for components of non-scalar random variables]
{\rm
Confidence bands for \textit{scalar} random variables are well-studied in the statistics literature \cite{ABO:95}. Their construction has been originally based on the Kolmogorov-Smirnov test \cite{ANK:33}, \cite{NVS:44}, for which rigorous confidence guarantees have been introduced in \cite{dvoretzky1956asymptotic} and further refined in \cite{massart1990tight}. A key difference of our approach is that we obtain analogous guarantees for an infinite (in fact uncountable) number of random variables, indexed by all spatio-temporal locations. This is achievable by using the Wasserstein ball guarantees in the finite-dimensional but in general \textit{non-scalar} parameter space. Therefore, resorting to traditional confidence band guarantees \cite{massart1990tight} is possible only in the restrictive case where we consider a single random parameter for the inputs. }
\oprocend
\end{rem}

We next present  explicit constructions for the upper and lower CDF envelopes of the empirical CDF.  For $n,m\in\bN$ and $t\in\Rat{}$, we use the conventions $[n:m]=\emptyset$ when $m<n$ and $[t,t)=\emptyset$. The proof of the following result is given in \cref{sec:app:to:sec:init:guarantees}.
 
\begin{proposition}[Upper CDF envelope for discrete distributions]\label{prop:envelopeconstruction}
Let $\widehat F\in\CD([a,b])$ be the CDF of a discrete distribution with positive mass $c_i$ at a finite number of points $t_i$, $i\in [1:N]$ satisfying $a=:t_0\le t_1<\dots<t_N\le b$ and define 
$
b_{i,j}:=\sum_{k=j}^i(t_k-t_j)c_k$, for $0\le j\le i\le N $, (with $b_{i,j}=0$ for any other $i,j\in\mathbb N_0$). Given $\rho>0$ with $b_{N,0}=\sum_{i=1}^N(t_i-a)c_i>\rho$, let $j_1:=0$, 
$i_1:=\min\{i\in[1:N]\,|\,b_{i,0}\ge\rho\}$ and
\begin{align*}
j_{k+1} & :=\max\{j\in[j_k:i_k]\,|\,b_{i_k,j}\ge\rho\}+1,\quad 
k=1,\ldots,k_{\max} \\
i_{k+1} & :=\min\{i\in[i_k+1:N]\,|\,b_{i,j_{k+1}}\ge\rho\},\quad 
k=1,\ldots,k_{\max}-1,  
\end{align*}	  
where $k_{\max}:=\min\{k\in\mathbb N\,|\,b_{N,j_{k+1}}\le\rho\}$. Then, all indices $j_k,i_k$ are well defined and   
\begin{align} \label{index:ordering}
j_k<j_{k+1}\le i_k<i_{k+1}\quad\forall k\in[1:k_{\max}], 
\end{align}
where $i_{k_{\max}+1}:= N+1$. Also, for each $k\in[1:k_{\max}]$, let 
\begin{align*}
\Delta t_\ell & :=\frac{\rho-b_{\ell,j_{k+1}}}{\sum_{l=j_{k+1}}^\ell c_l},\quad
\tau_\ell:=t_{j_{k+1}}-\Delta t_\ell,\quad \ell\in[i_k:i_{k+1}-1] \\  
\Delta y_\ell & :=\frac{\rho-b_{i_k-1,\ell}}{t_{i_k}-t_\ell}, \quad 
y_\ell:=\sum_{l=1}^{i_k-1}c_l+\Delta y_\ell, \quad \ell\in[j_k:j_{k+1}-1].
\end{align*} 
Then, $\tau_\ell$ are defined for all $\ell\in[i_1:N]$ and form a strictly increasing sequence with
\begin{align}
t_0 & =t_{j_1}\le\cdots\le t_{j_2-1}\le \tau_{i_1}\le\cdots\le\tau_{i_2-1}<t_{j_2}\le \cdots \label{all:times:ordering} \\
& \le t_{j_k}\le\cdots\le t_{j_{k+1}-1}\le \tau_{i_k}\le\cdots\le\tau_{i_{k+1}-1}<t_{j_{k+1}}\le \cdots \notag \\
& \le t_{j_{k_{\max}}}\le\cdots\le t_{j_{k_{\max}+1}-1}\le \tau_{i_{k_{\max}}}\le\cdots \notag \\
& \le\tau_{i_{k_{\max}+1}-1}=\tau_N<t_{j_{k_{\max}+1}}\le t_{i_{k_{\max}}}<t_N. \notag 
\end{align}
Further, the upper CDF envelope $\supscr{\widehat F}{up}\equiv\supscr{\mathcal F}{up}_\rho[\widehat F]$ of $\widehat F$ is given as 
\begin{align*}
& \supscr{\widehat F}{up}(t)=  \\
& \begin{cases}
0 & {\rm if}\; t\in(-\infty,a),  \\
z_\ell+(y_\ell-z_\ell)\frac{t_{i_k}-t_\ell}{t_{i_k}-t} & {\rm if}\; 
t\in[t_\ell,t_{\ell+1}),\ell\in[j_k:j_{k+1}-2], k\in[1:k_{\max}], \\
& {\rm if}\; t\in[t_{j_{k+1}-1},\tau_{i_k}),\ell=j_{k+1}-1,k\in[1:k_{\max}],\\
z_{j_{k+1}-1}+(z_\ell-z_{j_{k+1}-1})\frac{t_{\ell+1}-\tau_\ell}{t_{\ell+1}-t} &
{\rm if}\; t\in[\tau_\ell,\tau_{\ell+1}),\ell\in[i_k:i_{k+1}-2],k\in[1:k_{\max}], \\
&  {\rm if}\; t\in[\tau_{i_{k+1}-1},t_{j_{k+1}}),\ell=i_{k+1}-1,k\in[1:k_{\max}], \\
1 & {\rm if}\; t\in [\tau_N,\infty),
\end{cases}
\end{align*}
where $z_\ell:=\sum_{l=0}^\ell c_l$, $\ell\in[0:N]$ and  $c_0:=0$.
\end{proposition} 

\begin{figure}[tbh]
	\centering
	\includegraphics[width=.8\textwidth]{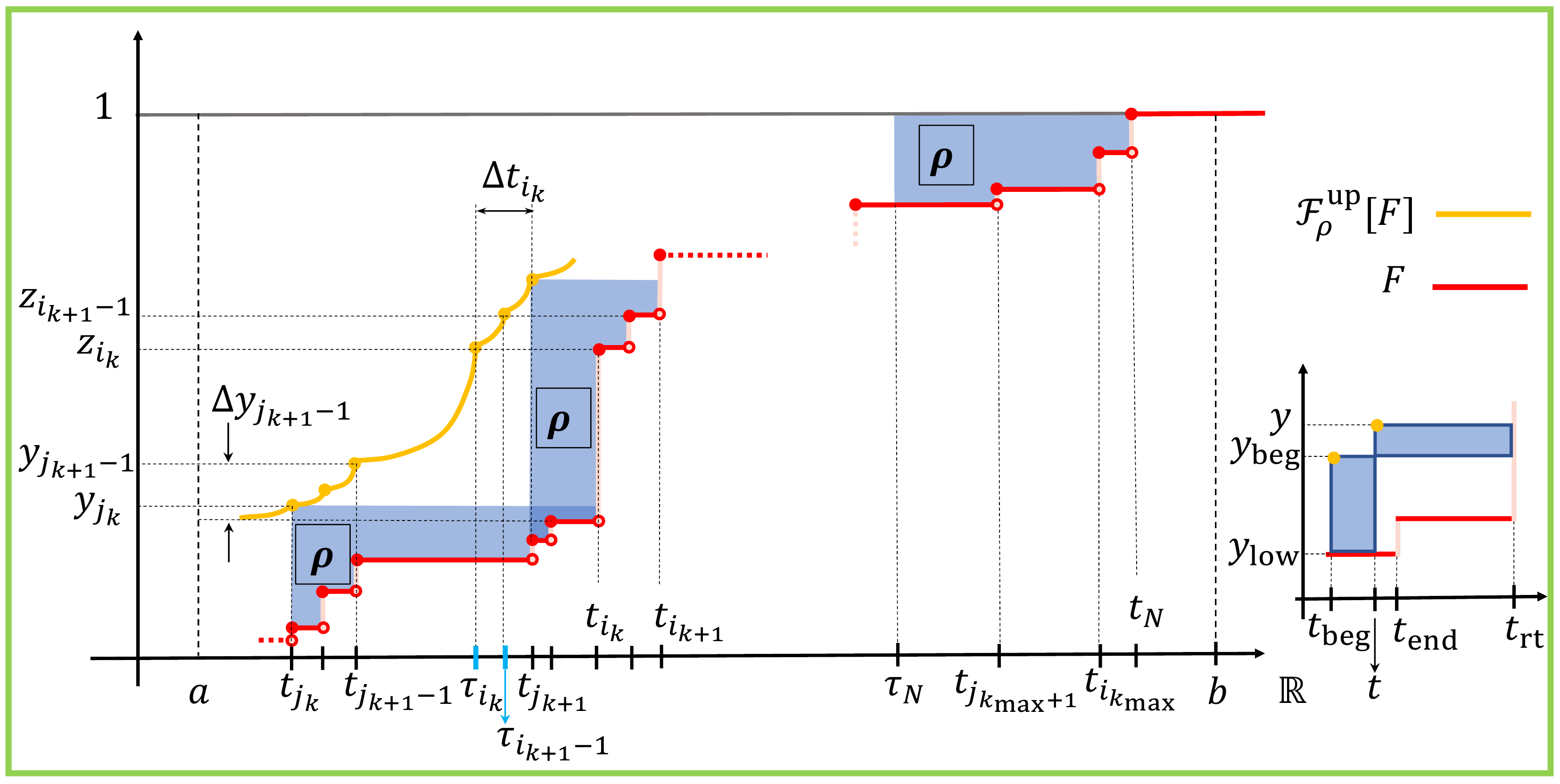}
	\caption{Illustration of how the upper CDF envelope $\supscr{\F}{up}_\rho[F]$ (in yellow) is constructed for a discrete distribution with a finite number of atoms.}\label{fig:envelope:empirical}
	\vspace*{-2ex}
\end{figure} 
 
\cref{prop:envelopeconstruction} is illustrated in \cref{fig:envelope:empirical}. 
To construct lower CDF envelopes, we introduce the reflection $\supscr{\F}{refl}_{(\frac{a+b}{2},\frac{1}{2})}[F]$ of a function $F$ around the point  $(\frac{a+b}{2},\frac{1}{2})$, i.e., 
$\supscr{\F}{refl}_{(\frac{a+b}{2},\frac{1}{2})}[F](t):=1-F(a+b-t)$, $t\in\Rat{}$. We also define the right-continuous version ${\rm rc}[G]$ of an increasing function $G$ by ${\rm rc}[G](t):=\lim_{s\searrow t}G(s)$, that satisfies $\int_a^tG(s)ds=\int_a^t{\rm rc}[G](s)ds$. Combining this with the fact that $G^{-1}\equiv ({\rm rc}[G])^{-1}$ when $G$ is increasing, we deduce from \cref{lemma:CDF:envelope} that the upper and lower CDF envelopes of a CDF $F$ are well defined and, in fact, are the 
same with those of any increasing function $G$ agreeing with $F$ everywhere except from its points of discontinuity, i.e., with ${\rm rc}[G]=F$.  
The next result explicitly constructs lower CDF envelopes by reflecting the upper CDF envelopes of reflected CDFs. Its proof is given in  \cref{sec:app:to:sec:init:guarantees}. 

\begin{lemma} [Lower CDF envelope via reflection]
\label{lemma:reflected:CDF:envelope}
Let $F\in\CD([a,b])$ and $\rho>0$ with $\rho\le\int_a^bF(t)dt$. Then, the lower CDF envelope of $F$  satisfies  
\begin{align*}
\supscr{\F}{low}_\rho[F]=\supscr{\F}{refl}_{(\frac{a+b}{2},\frac{1}{2})}\big[\supscr{\F}{up}_\rho\big[\supscr{\F}{refl}_{(\frac{a+b}{2},\frac{1}{2})}\big[F\big]\big]\big].
\end{align*}
\end{lemma}

Using \cref{lemma:reflected:CDF:envelope}, one can leverage \cref{prop:envelopeconstruction} to obtain the lower CDF envelope $\supscr{\mathcal F}{low}_\rho[F]$ of a  discrete distribution $F\in\CD([a,b])$ with mass $c_i>0$ at a finite number of points $a=: t_0\le t_1<\dots<t_N\le b$ for any $\rho>0$ with $\rho\le\int_a^bF(t)dt$.

\section{CDFs and 1-Wasserstein Distance propagation via the Method of Distributions}\label{sec:CDF:diff}

Here we develop the necessary tools to propagate in space and time the input ambiguity sets constructed in \cref{sec:init:guarantees}.
To obtain an evolution equation for the single-point cumulative distribution function $F_{u(\mathbf x,t)}$ of $u(\mathbf x,t)$, we introduce the random variable $\Pi(U,\mathbf x,t) = \mathcal H(U-u(\mathbf x,t))$,
parameterized by $U \in \mathbb R$. The ensemble mean of $\Pi$ over all possible realizations of $u$ at a point $(\mathbf x, t)$ is the single-point CDF 
\begin{align*}
    \langle \Pi(U,\mathbf x,t) \rangle  = F_{u(\mathbf x,t)}. 
\end{align*}
The dependence of $F_{u(\mathbf x,t)}$ on $U\in \mathbb R$ is implied. We henceforth use the notation $\widetilde \Omega \equiv \mathbb R\times\Omega $, $\widetilde \Gamma \equiv \mathbb R\times \Gamma $,  and $\widetilde {\mathbf x} \equiv (U, \mathbf x)\in \Rat{}\times\Rat{d}$. Using the Method of Distributions~\cite{tartakovsky2017method}, one can obtain the next result, whose derivation is  summarized in \cref{app:MD}.

\begin{thm}[Physics-driven CDF 
equation~\cite{boso-2014-cumulative}]\label{th:CDF:pushforward}
Let $F_{u_0(\mathbf x)}$, $\mathbf x\in\Omega$, and $F_{u_b(\mathbf x,t)}$, 
$(\mathbf x,t)\in\Gamma\times\RgeO$, be the CDFs of the initial and boundary conditions in \eqref{eq:ICsBCs}. Under \cref{assumption:exact:dynamics,assumption:unif:existence}, the CDF $F_{u(\mathbf x,t)}$ 
as a solution of \eqref{eq:tran} obeys 
\begin{equation}\label{eq:MD}
    \frac{\partial F_{u(\mathbf x,t)}}{\partial t} + \boldsymbol \Lambda \cdot \widetilde \nabla F_{u(\mathbf x,t)} = 0, \quad \widetilde{\mathbf x} \in \widetilde \Omega, t\in(0,T)
\end{equation}
with $\boldsymbol \Lambda = (\dot {\boldsymbol q} (U; \boldsymbol \theta_q), r(U; \boldsymbol \theta_r))$ and $\widetilde \nabla = (\nabla, \partial / \partial U) $, with $\dot{\mathbf q} = \partial \mathbf q / \partial U$, and subject to initial and boundary conditions $F_{u_0(\mathbf x)}$ and $F_{u_b(\mathbf x,t)}$, respectively. 
\end{thm}  


The CDF evolution is governed by the linear hyperbolic PDE \eqref{eq:MD}, which is specific for the physical model \eqref{eq:tran}. The next result exploits
the properties of \eqref{eq:MD} to obtain an upper bound across space and time on the difference between two CDFs.

\begin{corollary}[Propagation of upper bound on difference between CDFs]\label{th:eps:evolution}
Consider a pair of input CDFs $F_{u_0(\mathbf x)}^1$, 
$F_{u_0(\mathbf x)}^2$, $\mathbf x\in\Omega$, and $F_{u_b(\mathbf x,t)}^1$, 
$F_{u_b(\mathbf x,t)}^2$, $(\mathbf x,t)\in\Gamma\times\RgeO$ such that 
\begin{align}\label{eq:cond1}
| e_0(\widetilde{\mathbf x}) | & \ge | \varepsilon_0(\widetilde{\mathbf x}) | = | F_{u_0(\mathbf x)}^1 -F_{u_0(\mathbf x)}^2|,
\quad 
\forall \widetilde{\mathbf x}\in \widetilde \Omega \notag \\
| e_b(\widetilde{\mathbf x},t)| & \ge | \varepsilon_b(\widetilde{\mathbf x},t)| = | F_{u_b(\mathbf x,t)}^1 - F_{u_b(\mathbf x,t)}^2 |, \quad \forall (\widetilde{\mathbf x},t)\in\widetilde\Gamma\times\RgeO.
\end{align} 
Then, it holds that
\begin{align}\label{eq:ineq1}
|e(\widetilde{\mathbf x},t) |\ge | F_{u(\mathbf x,t)}^1 - F_{u(\mathbf x,t)}^2 | = |\varepsilon (\widetilde{\mathbf x},t)|,  \quad \forall (\widetilde{\mathbf x},t)\in 
\widetilde \Omega\times[0,T),
\end{align}
where $F_{u(\mathbf x,t)}^1$ and $F_{u(\mathbf x,t)}^2$ are the solutions 
of~\eqref{eq:MD} for the corresponding initial and boundary data, with $e(\widetilde{\mathbf x},t)$ obeying
\begin{align}\label{eq:disc4}
    & \frac{\partial |e|}{\partial t} + \boldsymbol \Lambda \cdot \widetilde \nabla |e| = 0, & \widetilde{\mathbf x} \in \widetilde \Omega, t>0\notag \\
    & |e(\widetilde{\mathbf x},t=0)| = |e_0(\widetilde{\mathbf x})|, & \widetilde{\mathbf x} \in \widetilde \Omega \notag \\
    & |e(\widetilde{\mathbf x},t)| = |e_b(\widetilde{\mathbf x},t)|, & \widetilde{\mathbf x} \in \widetilde \Gamma, t>0 
\end{align}
\end{corollary}
\begin{proof}
Exploiting the linearity of \eqref{eq:MD}, one can write an equation for the difference $\varepsilon (\widetilde{\mathbf x},t) = F^1_{u(\mathbf x,t)} - F^2_{u(\mathbf x,t)}$,
\begin{align}\label{eq:disc}
    & \frac{\partial \varepsilon}{\partial t} + \boldsymbol \Lambda \cdot \widetilde \nabla \varepsilon = 0, & \widetilde{\mathbf x} \in \widetilde \Omega, t\in(0,T) \notag \\
    & \varepsilon(\widetilde{\mathbf x},t=0) = \varepsilon_0(\widetilde{\mathbf x}), & \widetilde{\mathbf x} \in \widetilde \Omega \notag \\
    & \varepsilon(\widetilde{\mathbf x},t) = \varepsilon_b(\widetilde{\mathbf x},t), & \widetilde{\mathbf x} \in \widetilde\Gamma, t>0 
\end{align}
where $\varepsilon_0(\widetilde{\mathbf x}) = F^1_{u_0(\mathbf x)} - F^2_{u_0(\mathbf x)}$ and $\varepsilon_b(\widetilde{\mathbf x},t) = F^1_{u_b(\mathbf x,t)} - F^2_{u_b(\mathbf x,t)}$ are the initial and boundary differences, resp.
\eqref{eq:disc} can be expressed as the ODE system
$ \frac{\text d \varepsilon}{\text d s} = 0$, $\frac{\text d \widetilde{\mathbf x}}{\text d s} = \boldsymbol \Lambda$, $s>0$
with initial/boundary conditions assigned at the intersection between the characteristic lines and the noncharacteristic surface delimiting the space-time domain. 
Pointwise input differences $\varepsilon_0(\widetilde{\mathbf x})$ and $\varepsilon_b(\widetilde{\mathbf x},t)$ are conserved and propagate rigidly along deterministic characteristic lines, hence retaining the sign set by the input.
Since the system dynamics does not change the sign of $\varepsilon$ along the deterministic characteristic lines, $\varepsilon$ and $|\varepsilon|$ obey the same dynamics
\begin{align}\label{eq:disc3}
    & \frac{\partial |\varepsilon|}{\partial t} + \boldsymbol \Lambda \cdot \widetilde \nabla |\varepsilon| = 0, & \widetilde{\mathbf x} \in \widetilde \Omega, t\in(0,T)\notag \\
    & |\varepsilon(\widetilde{\mathbf x},t=0)| = |\varepsilon_0(\widetilde{\mathbf x})|, & \widetilde{\mathbf x} \in \widetilde \Omega \notag \\
    & |\varepsilon(\widetilde{\mathbf x},t)| = |\varepsilon_b(\widetilde{\mathbf x},t)|, & \widetilde{\mathbf x} \in \widetilde \Gamma \times \RgeO. 
\end{align}
For $e_0(\widetilde{\mathbf x},t)$ and $e_b(\widetilde{\mathbf x},t)$ as in \eqref{eq:cond1}, and $|e(\widetilde{\mathbf x},t)|$ obeying \eqref{eq:disc4}, \eqref{eq:disc3} implies \eqref{eq:ineq1}.
\end{proof}

The next result shows that propagation in space and time of CDFs is monotonic.

\begin{corollary}[Propagation of CDFs is monotonic]\label{th:eps:evolution2}
Consider a pair of input CDFs $F_{u_0(\mathbf x)}^1$, 
$F_{u_0(\mathbf x)}^2$, $\mathbf x\in\Omega$, and $F_{u_b(\mathbf x,t)}^1$, 
$F_{u_b(\mathbf x,t)}^2$, $(\mathbf x,t)\in\Gamma\times\RgeO$ such that 
\begin{align}\label{eq:cond3}
F_{u_0(\mathbf x)}^1 \ge F_{u_0(\mathbf x)}^2 & \quad 
\forall \widetilde{\mathbf x}\in \widetilde \Omega \notag \\
F_{u_b(\mathbf x,t)}^1 \ge F_{u_b(\mathbf x,t)}^2  &\quad \forall (\widetilde{\mathbf x},t)\in \widetilde\Gamma\times\RgeO
\end{align} 
Furthermore, we assume $F_{u(\mathbf x,t)}^1$ and $F_{u(\mathbf x,t)}^2$ to be solutions of \eqref{eq:MD} with $F_{u_0(\mathbf x)}^1, F_{u_b(\mathbf x,t)}^1$ and $F_{u_0(\mathbf x)}^2, F_{u_b(\mathbf x,t)}^2$ initial and boundary conditions, respectively. 
Then, it holds that 
\begin{align}\label{eq:ineq3}
F^1_{u(\mathbf x,t)} \ge F_{u(\mathbf x,t)}^2, \forall \widetilde{\mathbf x} \in \widetilde \Omega \times [0,T).
\end{align}
\end{corollary}
\begin{proof}
The discrepancy $\varepsilon(\widetilde{\mathbf x},t) = F_{u(\mathbf x,t)}^1 - F_{u(\mathbf x,t)}^2$ obeys \eqref{eq:disc}. Given non-negative initial and boundary conditions, consistently with \eqref{eq:cond3}, it holds that $\varepsilon(\widetilde{\mathbf x},t) \ge 0$ for all $\widetilde{\mathbf x} \in \widetilde \Omega, t\in(0,T)$, hence \eqref{eq:ineq3}.
\end{proof}

The CDF equation~\eqref{eq:MD} provides a computational tool for the space-time propagation of the CDFs of the inputs. 
If the governing equation \eqref{eq:tran} is linear, we show next that one can obtain an evolution equation in the form of a PDE for the 1-Wasserstein distance between each pair of distributions describing the same underlying physical process.

\begin{thm}[Physics-driven 1-Wasserstein discrepancy equation]
\label{th:WR:evol}
Consider a pair of distributions $F^1_{u(\mathbf x,t)}$ and $F^2_{u(\mathbf x,t)}$ obeying \eqref{eq:MD},
and assume linearity of \eqref{eq:tran}. Then, the 1-Wasserstein discrepancy between $F^1_{u({\mathbf x},t)}$ and $F^2_{u({\mathbf x},t)}$ defined by \eqref{eq:W1def}, $\omega_1 (\mathbf x,t) = \int_{\mathbb R} |F_{u(\mathbf x,t)}^1  - F_{u(\mathbf x,t)}^2 | \text d U$, obeys 
\begin{align}\label{eq:W1_PDE}
    & \frac{\partial \omega_1}{\partial t} + \dot {\boldsymbol q} \cdot \nabla \omega_1 - \dot r \; \omega_1 = 0, & {\mathbf x} \in \Omega, t>0\notag \\
    & \omega_1({\mathbf x},t=0) = \omega_0 (\mathbf x) , & {\mathbf x} \in \Omega \notag \\
    & \omega_1({\mathbf x},t) = \omega_b (\mathbf x,t), & {\mathbf x} \in \Gamma, t>0,
\end{align}
with  {$\omega_0(\mathbf x) = \int_{\mathbb R} |F_{u_0 (\mathbf x)}^1 - F_{u_0 (\mathbf x)}^2 | \text d 
U$ and $\omega_b = \int_{\mathbb R} |F_{u_b (\mathbf x,t)}^1 - F_{u_b (\mathbf x,t)}^2 | 
\text d U$} the input discrepancies.
\end{thm}  
\begin{proof}
\eqref{eq:W1_PDE} follows from \eqref{eq:disc4} by integration along $U \in \mathbb R$ assuming $F_{u(\mathbf x,t)}^1(U = \pm\infty) = F_{u(\mathbf x,t)}^2(U = \pm \infty)$, for all $\mathbf x \in \Omega,t>0$, accounting for the linearity of $\mathbf q(U)$ and $r(U)$. 
\end{proof}

\cref{th:eps:evolution} and the following  \cref{cor:W1:evolution} take advantage of the linearity and hyperbolic structure of \eqref{eq:disc4} and \eqref{eq:W1_PDE}, respectively, and identify a dynamic bound for the evolution of the pointwise CDF absolute difference and their 1-Wasserstein distance, respectively, once the corresponding discrepancies are set at the initial time and along the boundaries.

\begin{corollary}[Physics-driven 1-Wasserstein dynamic bound]
\label{cor:W1:evolution}
Consider the 
input CDF pairs  $F_{u_0(\mathbf x)}^1$, 
$F_{u_0(\mathbf x)}^2$, $\mathbf x\in\Omega$, and $F_{u_b(\mathbf x,t)}^1$, 
$F_{u_b(\mathbf x,t)}^2$, $(\mathbf x,t)\in\Gamma\times\RgeO$. Let $w(\mathbf x,t)$ 
be the solution of~\eqref{eq:W1_PDE} with initial and boundary conditions 
satisfying 
\begin{align}\label{eq:cond2}
w_0(\mathbf x) & \ge \omega_0(\mathbf x) = W_1\big(F_{u_0(\mathbf x)}^1,F_{u_0(\mathbf x)}^2\big) \quad 
\forall \mathbf x\in\Omega \notag \\
w_b(\mathbf x,t) & \ge \omega_b(\mathbf x,t) = W_1\big(F_{u_b(\mathbf x,t)}^1,F_{u_b(\mathbf x,t)}^2\big)\quad \forall (\mathbf x,t)\in\Gamma\times\RgeO.
\end{align} 
Then, it holds that
\begin{align}\label{eq:ineq2}
\omega_1(\mathbf x,t) = W_1\big(F_{u(\mathbf x,t)}^1,F_{u(\mathbf x,t)}^2\big)\le w(\mathbf x,t) \quad \forall (\mathbf x,t)\in\Omega\times\RgeO,
\end{align}
where $F_{u(\mathbf x,t)}^1$ and $F_{u(\mathbf x,t)}^2$ are the solutions 
of~\eqref{eq:MD} for the corresponding initial and boundary distributions.  
\end{corollary} 
\begin{proof}
\eqref{eq:ineq2} follows from condition \eqref{eq:cond2} and having 
$w(\mathbf x,t)$ and $\omega_1(\mathbf x,t)$ that fulfill \eqref{eq:W1_PDE} with conditions $w_0,w_b$ and $\omega_0,\omega_b$, respectively.
\end{proof}

\section{Ambiguity set propagation under finite-sample guarantees} \label{sec:ambiguity:evolution}

Here we combine the results from \cref{sec:init:guarantees,sec:CDF:diff} to build pointwise ambiguity sets for the distribution of $u(\mathbf x,t)$ 
over the whole spatio-temporal domain. We first consider the general PDE model \eqref{eq:tran} and study how the input ambiguity bands of \cref{corollary:CDF:ambig:sets} propagate in space and time using the CDF equation \eqref{eq:MD}.  

\begin{thm}[Ambiguity band evolution via the CDF dynamics] 
\label{thm:ambiguity:evolution} Assume that $N$ pairs of input samples are collected according to Assumption~\ref{assumption:sampling}. Consider a confidence $1-\beta$ and the CDFs
\begin{align*}
\supscr{F}{low}_{u_0(\mathbf x)} & :=\supscr{\F}{low}_{\rho_0(\mathbf 
x),[\alpha_0(\mathbf x),\gamma_0(\mathbf x)]}\big[\widehat F_{u_0(\mathbf 
x)}^N\big],\quad \mathbf x\in\Omega \\
\supscr{F}{low}_{u_b(\mathbf x,t)} & :=\supscr{\F}{low}_{\rho_b(\mathbf x,t),[\alpha_b(\mathbf x,t),\gamma_b(\mathbf x,t)]}\big[\widehat F_{u_b(\mathbf x,t)}^N\big],\quad (\mathbf x,t)\in\Gamma\times\RgeO \\
\supscr{F}{up}_{u_0(\mathbf x)} & :=\supscr{\F}{up}_{\rho_0(\mathbf 
x),[\alpha_0(\mathbf x),\gamma_0(\mathbf x)]}\big[\widehat F_{u_0(\mathbf 
x)}^N\big],\quad \mathbf x\in\Omega \\
\supscr{F}{up}_{u_b(\mathbf x,t)} & :=\supscr{\F}{up}_{\rho_b(\mathbf x,t),[\alpha_b(\mathbf x,t),\gamma_b(\mathbf x,t)]}\big[\widehat F_{u_b(\mathbf x,t)}^N\big], \quad (\mathbf x,t)\in\Gamma\times\RgeO,
\end{align*}
with $[\alpha_0(\mathbf x),\gamma_0(\mathbf x)]$, $[\alpha_b(\mathbf x,t),\gamma_b(\mathbf x,t)]$ and  $\rho_0(\mathbf x)$, $\rho_b(\mathbf x,t)$ as 
given in \eqref{interv:alpha:gamma:0}, \eqref{interv:alpha:gamma:b} and \eqref{rho:zero}, \eqref{rho:beta}, respectively. Let $\supscr{F}{low}_{u(\mathbf x,t)}$ and $\supscr{F}{up}_{u(\mathbf x,t)}$ be the solutions of~\eqref{eq:MD} with the corresponding input CDFs above and define the ambiguity sets \begin{align*} 
\P_{\mathbf x,t}^{\Env}:= \big\{F\in \CD(\Rat{})\,|\,\supscr{F}{low}_{u(\mathbf x,t)} \le F\le \supscr{F}{up}_{u(\mathbf x,t)} \;\forall U\in\Rat{}\big\}, \quad 
\mathbf x\in\Omega, t\in[0,T).  
\end{align*}
Then $\bP(\supscr{F}{true}_{u(\mathbf x,t)}\in \P_{\mathbf x,t}^{\Env} \;\forall 
(\mathbf x,t)\in\Omega\times[0,T)) \ge 1-\beta$.
\end{thm} 
\begin{proof}
Let
\begin{align*}
A:=\{\mathbf (\mathbf a^1,\ldots,\mathbf a^N)\in\Rat{Nn}\,|\, & \supscr{F}{true}_{u_0(\mathbf x)}\in \P_{\mathbf x}^{0,\Env}(\mathbf a^1,\ldots,\mathbf a^N)\;\forall \mathbf x\in\Omega \\
& \land\;\supscr{F}{true}_{u_b(\mathbf x,t)}\in \P_{\mathbf x,t}^{b,\Env}(\mathbf a^1,\ldots,\mathbf a^N)\;\forall (\mathbf x,t)\in\Gamma\times\RgeO\},
\end{align*}
with $\P_{\mathbf x}^{0,\Env}$ and $\P_{\mathbf x,t}^{b,\Env}$ as given in \cref{corollary:CDF:ambig:sets}, where we emphasize their dependence on the parameter realizations. Then, we have from \eqref{guarantee:init:bnd:env} that  
\begin{align} \label{event:A:probability}
\bP((\mathbf a^1,\ldots,\mathbf a^N)\in A)\ge 1-\beta.
\end{align}
Next, let $(\mathbf a^1,\ldots,\mathbf a^N)\in A$ and $\widehat F_{u_0(\mathbf x)}^N\equiv\widehat F_{u_0(\mathbf x)}^N(\mathbf a^1,\ldots,\mathbf a^N)$, $\mathbf x\in\Omega$, $\widehat F_{u_b(\mathbf x,t)}^N\equiv\widehat F_{u_b(\mathbf x,t)}^N(\mathbf a^1,\ldots,\mathbf a^N)$, $(\mathbf x,t)\in\Gamma\times\RgeO$ be the associated empirical input CDFs. These generate the corresponding lower CDF envelopes $\supscr{F}{low}_{u_0(\mathbf x)}\equiv \supscr{F}{low}_{u_0(\mathbf x)}(\mathbf a^1,\ldots,\mathbf a^N)$ and $\supscr{F}{low}_{u_b(\mathbf x,t)}\equiv \supscr{F}{low}_{u_b(\mathbf x,t)}(\mathbf a^1,\ldots,\mathbf a^N)$ given in the statement, and we deduce from the definitions of $A$ and the ambiguity sets $\P_{\mathbf x}^{0,\Env}$, $\P_{\mathbf x,t}^{b,\Env}$ that $
\supscr{F}{true}_{u_0(\mathbf x)}(U)  \ge \supscr{F}{low}_{u_0(\mathbf x)}(U)$ for all $U\in\Rat{},\mathbf x\in\Omega$ and $
\supscr{F}{true}_{u_b(\mathbf x,t)}(U)  \ge \supscr{F}{low}_{u_b(\mathbf x,t)}(U)$ for all $U\in\Rat{}, (\mathbf x,t)\in\Gamma\times\RgeO.
$
Thus, we obtain from \cref{th:eps:evolution2} applied with $F_u^1\equiv 
\supscr{F}{true}_u$ and $F_u^2\equiv\supscr{F}{low}_u$ that 
\begin{align*} 
\supscr{F}{true}_{u(\mathbf x,t)}(U) & \ge \supscr{F}{low}_{u(\mathbf x,t)}(U) \quad\forall U\in\Rat{}, (\mathbf x,t)\in\Omega\times[0,T).
\end{align*}
Analogously, we get that 
$\supscr{F}{true}_{u(\mathbf x,t)}(U)  \le \supscr{F}{up}_{u(\mathbf x,t)}(U)$ for all $U\in\Rat{}, (\mathbf x,t)\in\Omega\times[0,T)$,
and we deduce from the definition of the ambiguity sets $\P_{\mathbf x,t}^{\Env}$ in the statement that
\begin{align*}
\supscr{F}{true}_{u(\mathbf x,t)}\in \P_{\mathbf x,t}^{\Env}(\mathbf 
a^1,\ldots,\mathbf a^N)\quad \forall U\in\Rat{}, (\mathbf x,t)\in\Omega\times[0,T).
\end{align*}  
The result now follows from~\eqref{event:A:probability}.
\end{proof}

Under linearity of the dynamics, we can exploit \cref{cor:W1:evolution} to propagate the tighter Wasserstein input ambiguity balls of \cref{prop:init:bnd:ambiguity:sets}.  

\begin{thm}[Ambiguity set evolution for linear dynamics] 
\label{eq:th:guarantees}
Assume that PDE \eqref{eq:tran} is linear and $N$ pairs of input  samples are collected according to \cref{assumption:sampling}. Consider a confidence level $1-\beta$ and let $w(\mathbf x,t)$ be the solution of~\eqref{eq:W1_PDE} with $
w_0(\mathbf x) = L_0(\mathbf x) \epsilon_N(\beta,\rho_{\mathbf a})$, $ 
\mathbf x\in\Omega$ and $w_b(\mathbf x,t)  = L_b(\mathbf x,t) \epsilon_N(\beta,\rho_{\mathbf a})$, $(\mathbf x,t)\in\Gamma\times\RgeO$, and $L_0(\mathbf x)$, $L_b(\mathbf x,t)$, $\rho_{\mathbf a}$, and  $\epsilon_N(\beta,\rho_{\mathbf a})$ given by~\eqref{Lip:constant:init}, \eqref{Lip:constant:bnd}, \eqref{rho:a}, and~\eqref{epsN:dfn}. Let $\widehat F_{u(\mathbf x,t)}^N$ be the 
solution of~\eqref{eq:MD} with the empirical input CDFs $\widehat F_{u_0(\mathbf x)}^N$ and $\widehat F_{u_b(\mathbf x,t)}^N$ as given in \cref{sec:init:guarantees} and define the ambiguity sets 
\begin{align*} 
\P_{\mathbf x,t}:= \big\{F\in \CD(\Rat{})\,|\,W_1(\widehat F_{u(\mathbf 
x,t)}^N,F)\le w(\mathbf x,t)\big\}, \quad \mathbf x\in\Omega, t\in\RgeO.  
\end{align*}
Then $\bP(\supscr{F}{true}_{u(\mathbf x,t)}\in \P_{\mathbf x,t} \;\forall (\mathbf x,t)\in\Omega\times\RgeO) \ge 1-\beta$.
\end{thm} 
\begin{proof}
Let
$A:=\{\mathbf (\mathbf a^1,\ldots,\mathbf a^N)\in\Rat{Nn}\,|\,  \supscr{F}{true}_{u_0(\mathbf x)}\in \P_{\mathbf x}^0(\mathbf a^1,\ldots,\mathbf a^N)\;\forall \mathbf x\in\Omega \;\land\;\supscr{F}{true}_{u_b(\mathbf x,t)}\in \P_{\mathbf x,t}^b(\mathbf 
a^1,\ldots,\mathbf a^N)\;\forall (\mathbf x,t)\in\Gamma\times\RgeO\}$,
with $\P_{\mathbf x}^0$ and $\P_{\mathbf x,t}^b$ as given in \cref{prop:init:bnd:ambiguity:sets}. Then, we have from \eqref{guarantee:init:bnd} that \eqref{event:A:probability} holds. Next, let $(\mathbf a^1,\ldots,\mathbf a^N)\in A$ and $\widehat F_{u_0(\mathbf x)}^N\equiv\widehat F_{u_0(\mathbf x)}^N(\mathbf a^1,\ldots,\mathbf a^N)$, $\mathbf x\in\Omega$,  $\widehat F_{u_b(\mathbf x,t)}^N\equiv\widehat F_{u_b(\mathbf x,t)}^N(\mathbf a^1,\ldots,\mathbf a^N)$, $(\mathbf x,t)\in\Gamma\times\RgeO$ be the associated input CDFs. From the definition of $\P_{\mathbf x}^0$, $\P_{\mathbf x,t}^b$ and $w_0$, $w_b$ we get 
\begin{align*}
W_1\big(\widehat F_{u_0(\mathbf x)}^N,\supscr{F}{true}_{u_0(\mathbf x)}\big) & \le w_0(\mathbf x) \quad\forall \mathbf x\in\Omega\\
W_1\big(\widehat F_{u_b(\mathbf x,t)}^N,\supscr{F}{true}_{u_b(\mathbf x,t)}\big) & \le w_b(\mathbf x,t) \quad\forall (\mathbf x,t)\in\Gamma\times\RgeO.
\end{align*}
Thus, applying \cref{cor:W1:evolution} with $F^1\equiv 
\widehat F_u^N$ and $F^2\equiv\supscr{F}{true}_u$,
$W_1\big(\widehat F_{u(\mathbf x,t)}^N,\supscr{F}{true}_{u(\mathbf x,t)}\big) \le w(\mathbf x,t)$, for all $(\mathbf x,t)\in\Omega\times\RgeO$,
and it follows from the definition of $\P_{\mathbf x,t}$ that
\begin{align*}
\supscr{F}{true}_{u(\mathbf x,t)}\in \P_{\mathbf x,t}(\mathbf 
a^1,\ldots,\mathbf a^N)\quad\forall (\mathbf x,t)\in\Omega\times\RgeO.
\end{align*}  
Combining this with \eqref{event:A:probability} for $A$ as given in this proof yields the result.
\end{proof}

\section{Numerical example} \label{sec:numerical:examples} 

In this section, we illustrate the use of the ambiguity propagation tools developed above in a numerical example. 
We consider a one-dimensional version of \eqref{eq:tran} with linear 
\begin{align}\label{eq:tranexample}
q(u) = u, \quad \text{and} \quad r(u; \theta_r) = \theta_r u, \quad \theta_r \in \mathbb R,
\end{align}
defined in $\Omega = \mathbb R_{\ge 0}$ and subject to the following initial and boundary conditions 
\begin{align}\label{eq:exampleinputs}
u(x,0) & = u_0 = a_1 + a_2,\quad x\ge 0 \notag \\
u(0,t) & = u_b(t) = a_1+ a_2 \left( 1 +a_3 \sin(2 \pi t) \right) ,\quad t\ge 0 
\end{align}
%
(note that this fulfills the most restrictive conditions of \cref{th:WR:evol}).
Because of \eqref{eq:exampleinputs}, in the following we drop the dependence of the input and boundary conditions from $x$. Randomness is introduced by the finite set of $(n=3)$ i.i.d. uncertain parameters $\mathbf a=(a_1,a_2,a_3)$, which vary in $[0,1]^n$; according to \eqref{rho:a}, $\rho_{\mathbf a} = 1/2$. We choose a uniform distribution to be the data-generating distribution for~$\mathbf a$.  
Both $u_0$ and $u_b(t)$ are random non-negative variables which are defined on the compact supports $\left[ 0,2 \right]$ and $\left[ 0, 2 + \max \left( 0, \sin \left( 2 \pi t \right) \right) \right]$, respectively. 

\subsection{Shape and size of the input ambiguity sets}

We consider data-driven 1-Wasserstein ambiguity sets for the parameters $\mathbf a$, which are constructed according to \cref{lemma:ambiguity:radius} using $p=1$ and $n=3$. We choose the radius $\epsilon_N(\beta,\rho_{\mathbf a})$ in \eqref{epsN:dfn} for a given sample size $N$ and a fixed $\beta$. 
%
%
Threshold radii for different size of the sample $N$ and identical confidence level $1-\beta$ can be constructed in relative terms, as exemplified  in~\cite{DB-JC-SM:19}. By adjusting 
$\epsilon_N(\beta,\rho_{\mathbf a})$, the decision-maker determines the level of conservativeness of the ambiguity set, and the distributional robustness as a consequence.
The ambiguity sets for the parameters are scaled into pointwise ambiguity sets for the inputs following \cref{prop:init:bnd:ambiguity:sets}, via the definition of the Lipschitz constants 
\begin{align}\label{eq:rhos}
& \rho_0 = L_0 \epsilon_N(\beta,\rho_{\mathbf a}), & \text{with} \; & L_0: = \sqrt{2}, \notag \\ 
& \rho_b(t) = L_b(t) \epsilon_N(\beta,\rho_{\mathbf a}), & \text{with} \; & L_b(t) := \sqrt{2+ 2 \sin^2(2 \pi t) + 2 \max(0,\sin(2 \pi t))}. 
\end{align}
Second, we construct conservative ambiguity envelopes for the initial and the boundary conditions characterized by a 1-Wasserstein discrepancy larger than $\rho_0$ and $\rho_b(t)$, respectively, according to \cref{prop:envelopeconstruction}. These upper and lower envelopes define an ambiguity band which enjoys the same performance guarantees as the previously defined 1-Wasserstein ambiguity sets. We denote with $\rho_0^{\Env} \ge \rho_0$ and $\rho_b^{\Env}(t) \ge \rho_b(t)$ the 1-Wasserstein discrepancy between the upper and lower distributions defining the initial and boundary ambiguity bands, respectively.

For both inputs, the maximum pointwise Wasserstein distance $\rho_{0,\max}$ and $\rho_{b,\max}(t)$ corresponds to the local size of the support. 1-Wasserstein discrepancies larger than the maximum value denote uniformative ambiguity sets.
For the chosen scenario, $\rho_{0,\max} = 2$ and $\rho_{b,\max}(t) = 2+\max(0,\sin(2\pi t))$ for the initial and the boundary values, respectively.
A comparison of $\rho_b(t)$, $\rho_b^{\Env}(t)$ and $\rho_{b,\max}(t)$ is presented in \cref{fig:W1b} for different sample sizes $N$ and identical confidence level $1-\beta$. 
The corresponding values for the initial condition can be read in the same figure at $t=0$ because of the imposed continuity between initial and boundary conditions at $t=0$. 
Regardless of the chosen shape of the ambiguity set, larger $N$ 
determines smaller ambiguity sets characterized by smaller 1-Wasserstein discrepancies. 
By construction, 1-Wasserstein ambiguity sets defined through~\eqref{eq:rhos} are sharper than the corresponding ambiguity bands drawn geometrically via \cref{prop:envelopeconstruction} at all times. The temporal behavior of $\rho_b(t)$ is determined by the Lipschitz scaling function $L_b(t)$ in~\eqref{eq:rhos}; in this case it is periodic and bounded.
\begin{figure}[tbh]
	\centering
	\includegraphics[width=.75\textwidth]{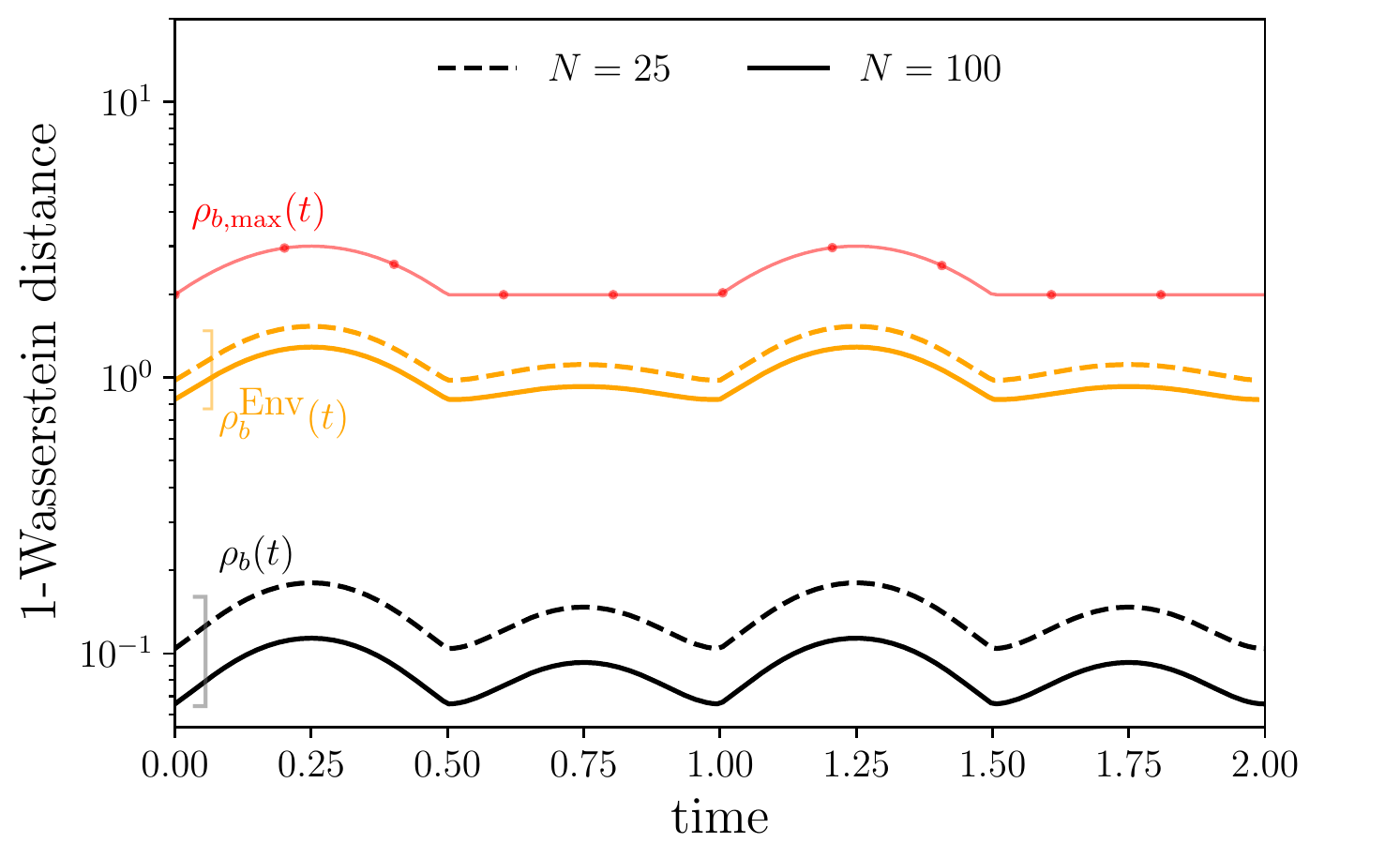}
	\caption{{Characteristic 1-Wasserstein distances for the pointwise ambiguity sets for $u_b(0,t)$. Black lines correspond to the $\rho_b(t)$ bounds set in \cref{corollary:CDF:ambig:sets} and used to define 1-Wasserstein ambiguity sets. Yellow lines indicate $\rho_{b}^{\text{Env}}(t)$, the sample-dependent 1-Wasserstein discrepancy between envelopes defined via the \cref{prop:envelopeconstruction} procedure. The line pattern indicates the size of the data sample $N$, as listed in the legend. The maximum theoretical 1-Wasserstein discrepancy for $u_b(0,t)$, $\rho_{b,\max}(t)$, is also drawn (red circles).} 
	}\label{fig:W1b}
	\vspace*{-2ex}
\end{figure}
Figures~\ref{fig:W10_env} and~\ref{fig:W1b_env} show the corresponding ambiguity bands for $u_0$ and $u_b(t)$ at a given time $t$, respectively, for the same values of sample size $N$ and identical confidence level $1-\beta$. 
Both upper and lower envelopes are data-driven, i.e., they depend on the empirical distribution of a specific sample. We also show  the 1-Wasserstein discrepancy between the upper and lower envelopes. 

\begin{figure}\label{fig:W10_env}
	\centering
	\includegraphics[width=.75\textwidth]{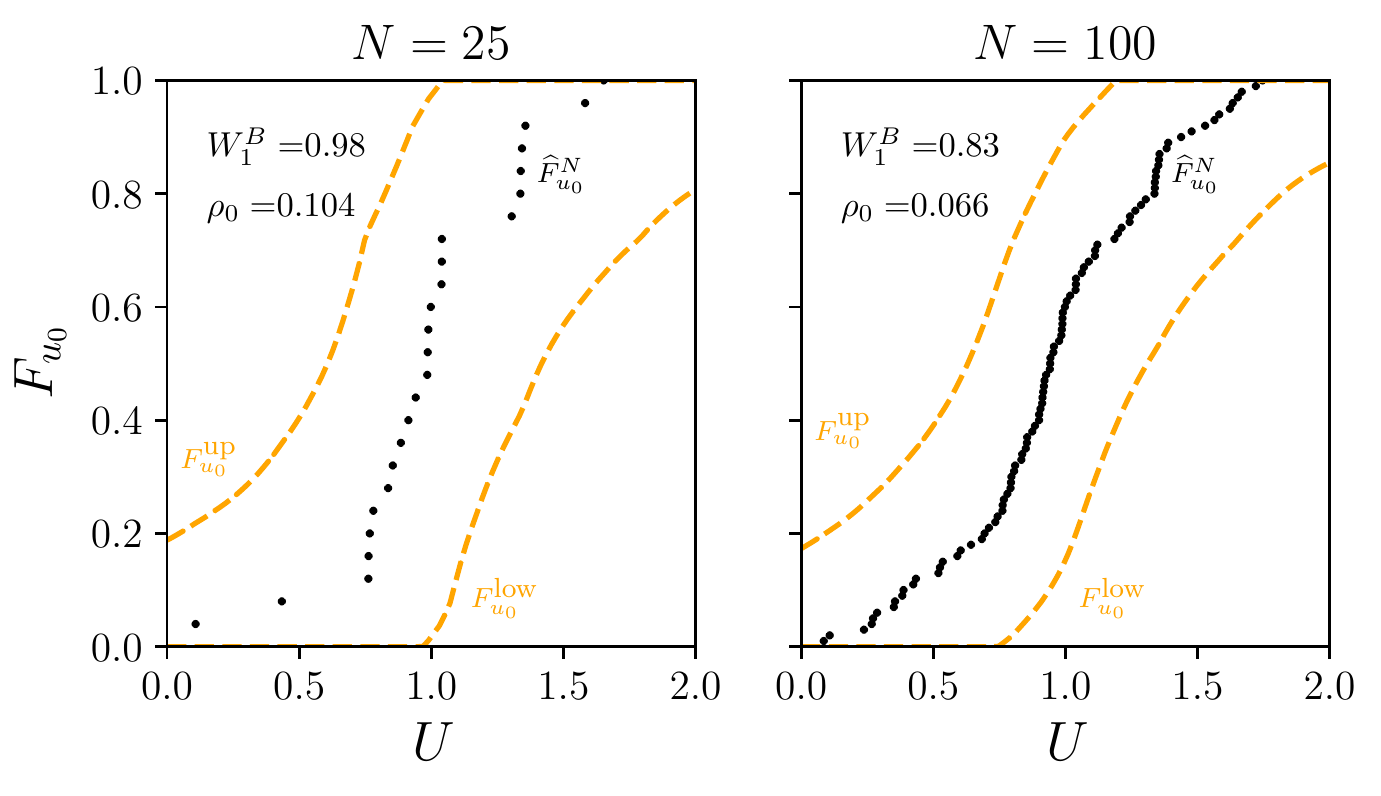}
	\caption{Ambiguity band for the distributions of $u_0$ for different sample size $N$ and identical confidence level $1-\beta$. We use  $\theta_r=-1$. Scatter points represent the empirical distribution $\widehat F^N_{u_0}$. Dashed yellow lines represent the conservative envelopes (with respect to a minimum 1-Wasserstein distance $\rho_0$) constructed according to \cref{prop:envelopeconstruction}. 
	The 1-Wasserstein discrepancies for the ambiguity band - computed between the upper and the lower envelope - are reported in the corresponding panels, also indicating $\rho_0$. }
\end{figure}

\begin{figure}\label{fig:W1b_env}
	\centering
	\includegraphics[width=.75\textwidth]{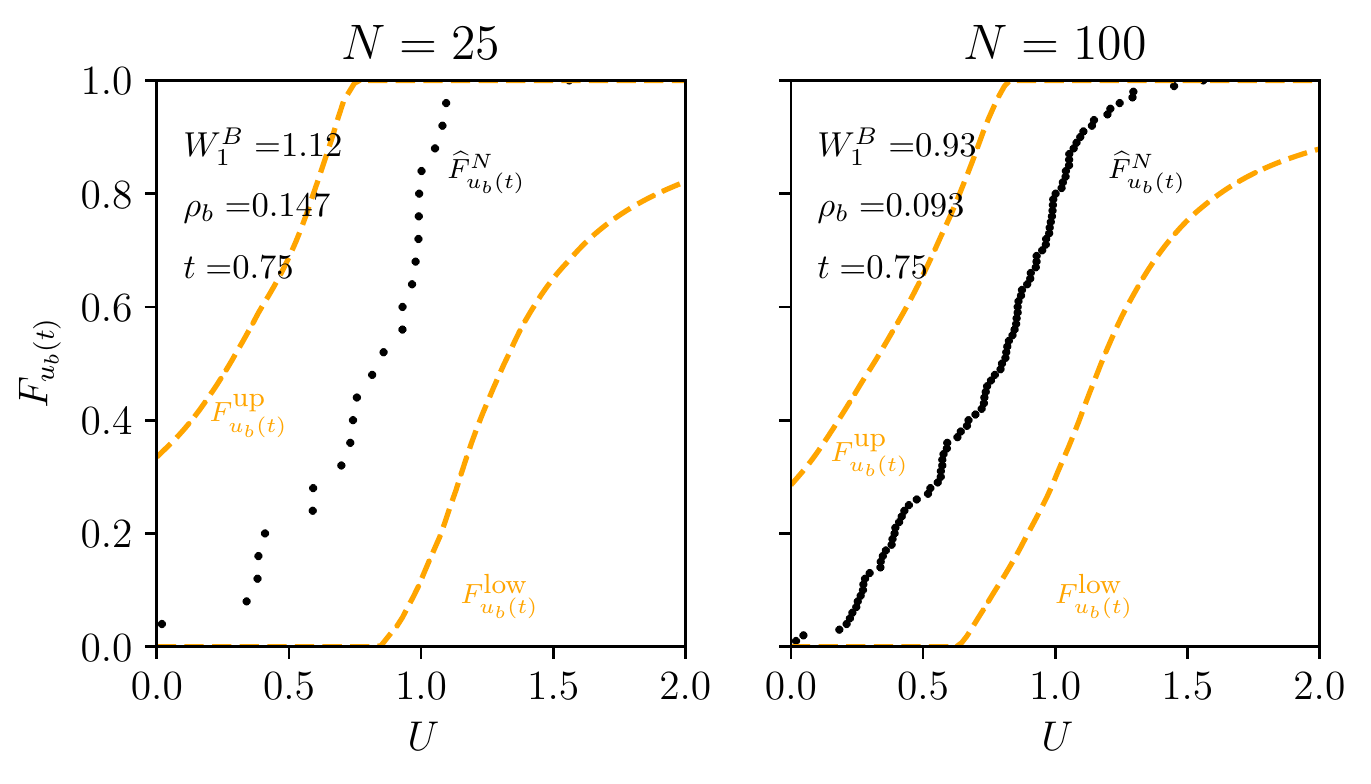}
	\caption{Ambiguity band for the distributions of $u_b(t)$ at $t=0.75$ for different sample size $N$ and identical confidence level $1-\beta$. We use $\theta_r=-1$. Scatter points represent the empirical distribution $\widehat F^N_{u_b(t)}$. Dashed yellow lines represent the conservative envelopes (with respect to a minimum 1-Wasserstein distance $\rho_b(t)$) constructed according to \cref{prop:envelopeconstruction}. The 1-Wasserstein discrepancies for the ambiguity band are reported in the corresponding panels, also indicating $\rho_b(t)$.}
	\vspace*{-2ex}
\end{figure}

\subsection{Propagation of the ambiguity set}

Pointwise 1-Wasserstein distances for the inputs can be propagated in space and in time to describe the behavior of the ambiguous distributions using \eqref{eq:W1_PDE}, under the assumption of linear dynamics. Solving \eqref{eq:W1_PDE} yields a quantitative measure of the stretch/shrink of the ambiguity ball in each space-time location. True (unknown) distributions as well as their empirical approximations describing the given physical dynamics evolve according to \eqref{eq:MD}; the latter provide an anchor for the pointwise ambiguity balls in $(\mathbf x,t)$. In \cref{fig:W1_WAS} we present the solution of \eqref{eq:W1_PDE}, $w_1(x,t)$, solved using $\rho_0$ and $\rho_b(t)$ as defined in \eqref{eq:rhos} as initial and boundary conditions, respectively. The ambiguity ball shrinks with respect to the input conditions as an effect of a depletion dynamics imposed by \eqref{eq:tran} with the given choice of $\theta_r=-1$. As expected, the smaller the sample size $N$, the larger the radius of the ambiguity ball as quantified by $w_1(x,t)$.  

\begin{figure}[htb]\label{fig:W1_WAS}
	\centering
	\includegraphics[width=.75\textwidth]{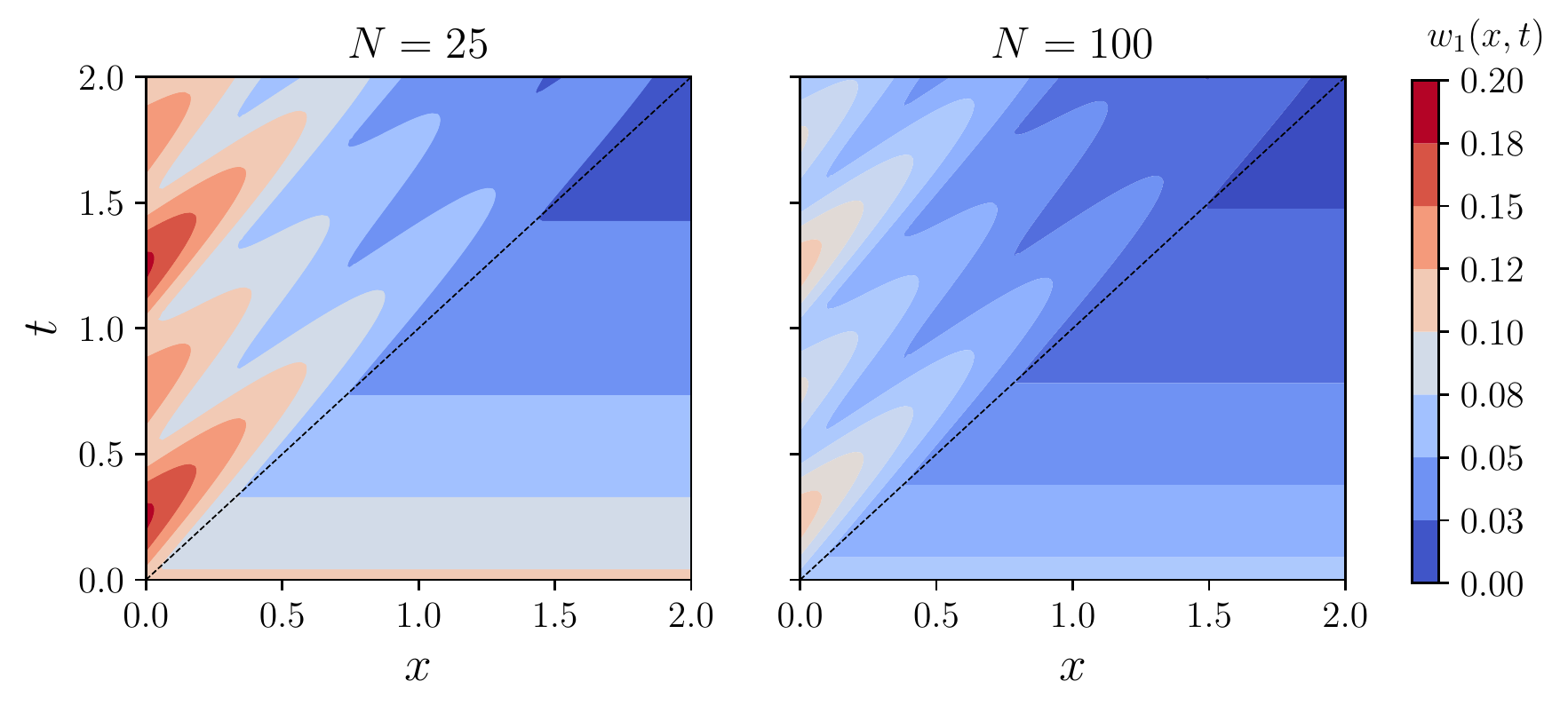}
	\caption{$w_1(x,t)$ as a solution of \eqref{eq:W1_PDE} with $w_0(x) = \rho_0$ and $w_b(x,t) = \rho_b(t)$ for different sample size $N$ ($N=25$ in the left panel, and $N=100$ in the right panel)  and identical confidence level $1-\beta$. The dotted line represents the domain partition between regions where information originates from either the initial or the boundary condition. We use $\theta_r=-1$.}
\end{figure}

The dynamic evolution of ambiguity bands is determined by the evolution of the upper and lower envelopes for the input samples, cf. \cref{prop:envelopeconstruction}, for given sample size $N$ and confidence level $1-\beta$. The envelopes evolve according to \eqref{eq:MD}, thus requiring no linearity assumption for \eqref{eq:tran}. {As such, ambiguity bands, while possibly being more conservative than 1-Wasserstein ambiguity sets in terms of size, can be evolved for a wider class of hyperbolic equations.} Ambiguity bands are equipped with 1-Wasserstein measures, as the 1-Wasserstein distance between the upper and the lower envelope represents the maximum distance between any pair of distributions within the band, and it is constructed to be always larger or equal than the local radius of the corresponding ambiguity ball. Confidence guarantees established for the inputs (\cref{corollary:CDF:ambig:sets}) withstand propagation, as demonstrated in \cref{thm:ambiguity:evolution}. 

For a given choice of $N$, 
we compare the propagation of 1-Wasserstein ambiguity sets with input conditions defined by \eqref{eq:rhos} to the data-driven dynamic ambiguity bands constructed via \cref{prop:envelopeconstruction} and subject to the input envelopes represented in Figures~\ref{fig:W10_env} and~\ref{fig:W1b_env}. The corresponding $w_1$ maps are shown in \cref{fig:W1_2locs} (top row). In both cases, the pointwise 1-Wasserstein distance undergoes the same dynamics established by \eqref{eq:W1_PDE}, but subject to different inputs (represented in \cref{fig:W1b}). In each spatial location, it is possible to track the temporal behavior of the ambiguity set size for both shapes, as shown for two representative locations in \cref{fig:W1_2locs} (bottom row). The size of both ambiguity sets decreases from the maximum imposed at the initial time for $t<x$, and reflects the temporal signature of the boundary, dampened as an effect of depletion dynamics introduced by \eqref{eq:tranexample} with $\theta_r=-1$, for $t>x$.

\begin{figure}[htb]
	\centering
	\includegraphics[width=.85\linewidth]{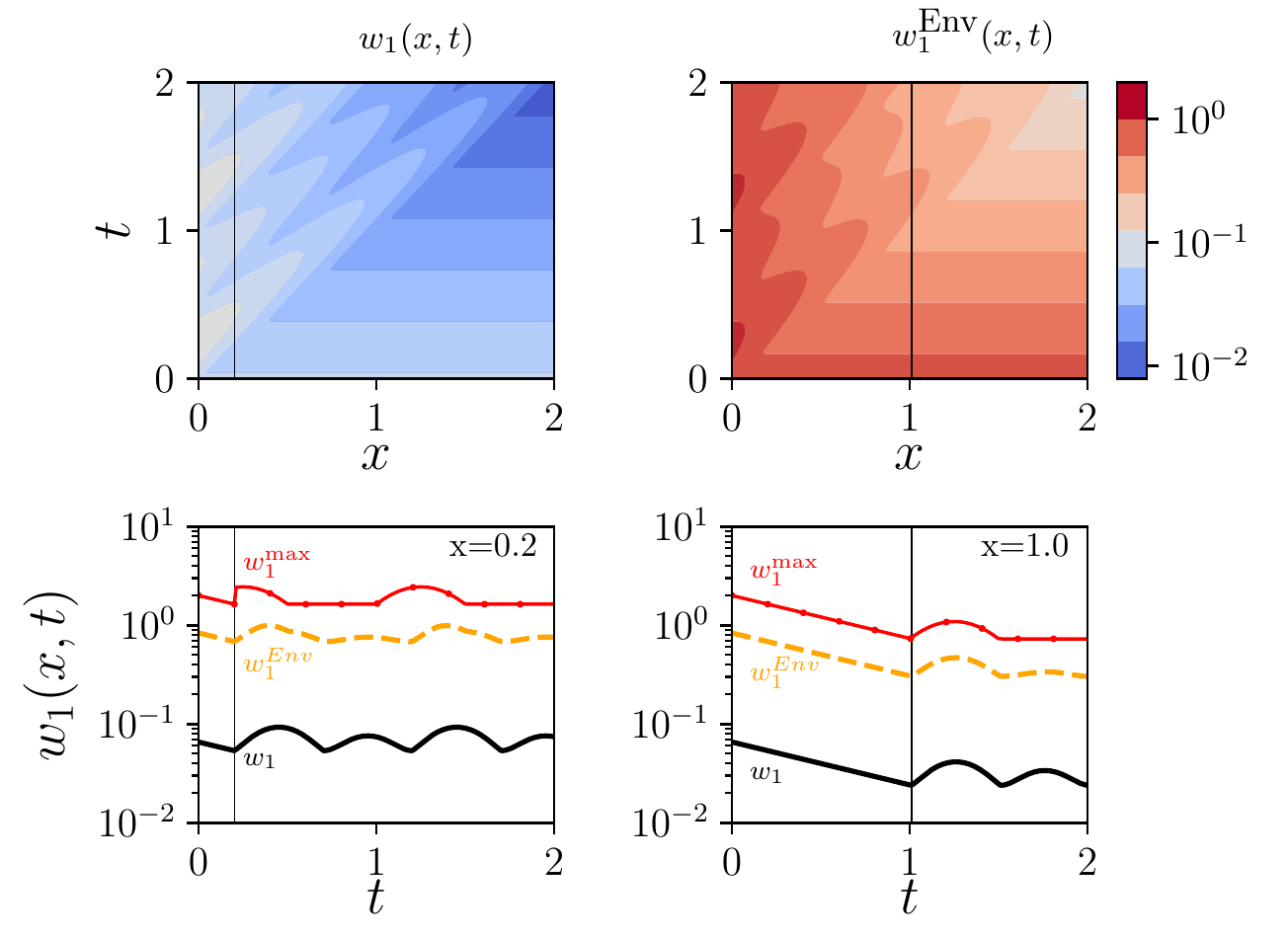}
	\caption{{Top row: 1-Wasserstein distance maps for the radius of the ambiguity balls $w_1(x,t)$ with input radii \eqref{eq:rhos} (left), and the ambiguity band $w_1^{\Env} (x,t)$ (right), where $w_1^{\Env} (x,t) = W_1(F_{u(x,t)}^{\text{low}},F_{u(x,t)}^{\text{up}})$. Bottom row: 1-Wasserstein distance profiles at given locations $x=\{0.2,1.\}$. The black solid line reflects the 1-Wasserstein ambiguity radius $w_1(x,t)$, whereas the yellow dashed line represents the 1-Wasserstein distance of the ambiguity band, $w_1^{\Env}(x,t)$. The maximum theoretical 1-Wasserstein discrepancy is also drawn, $w_1^{\max}(x,t)$ (marked red line). The location of the cross-sections is indicated in the top-row contour plots in the corresponding column ($x=0.2$ and $x=1$, respectively), whereas the demarcation line $t=x$ is indicated in the bottom panels. Parameters are set to: $N=100$, $\theta_r=-1$.}}\label{fig:W1_2locs}
	\vspace*{-2ex}
\end{figure}


\section{Conclusions}

We have provided computational tools in the form of PDEs for the space-time propagation of 
pointwise ambiguity sets for random variables obeying hyperbolic conservation laws. 
The initial and boundary conditions of these propagation PDEs depend on the data-driven characterization of the ambiguity sets at the initial time and along the physical boundaries of the spatial domain. We have introduced both 1-Wasserstein ambiguity balls and ambiguity bands, formed through upper and lower CDF envelopes containing all distributions with an assigned 1-Wasserstein distance from their empirical CDFs.  
The former are propagated by evolving the ambiguity radius according to a dynamic law that can be derived exactly in the case of linear physical models. The latter are propagated by solving the CDF equation
for both the upper and the lower CDF envelope defining the ambiguity band. 
In this second case, both linear and non linear physical processes can be described exactly in CDF terms, provided that no shock develops in the physical model solution. 
The performance guarantees for the input ambiguity sets of both types are demonstrated to withstand propagation through the physical dynamics. These computational tools allow the modeler to map the physics-driven stretch/ shrink of the ambiguity sets size, enabling dynamic evaluations of distributional robustness.
Future research will consider systems of conservation laws with joint one-point CDFs, the characterization of ambiguity sets when shocks are formed under nonlinear dynamics, the assimilation of data collected within the space-time domain, the application of these results in distributionally robust optimization problems, and sharper concentration-of-measure results to reduce conservativeness of the ambiguity sets for small numbers of samples.  

\appendix

\section{Technical proofs from Section~\ref{sec:init:guarantees}} \label{sec:app:to:sec:init:guarantees}
We collect here basic properties of generalized CDF inverses used in the following: 

\noindent \textbf{(GI1)} $F(t)< y\Rightarrow t<F^{-1}(y)$;

\noindent \textbf{(GI2)} $F(t_1)\le y\le F(t_2)\Rightarrow t_1\le F^{-1}(y)\le t_2$;

\noindent \textbf{(GI3)} $t<F^{-1}(y)\Rightarrow F(t)<y$;

\noindent \textbf{(GI4)} $F(t)=F(t_1)\;\forall t\in[t_1,t_2)\land F(t_1)<y\le F(t_2)\Rightarrow F^{-1}(y)=t_2$.

\begin{proof}[Proof of Lemma~\ref{lemma:Lipschitz:map}]
Let $\widehat T:\Rat{n}\times\Rat{n}\to\Rat{m}\times\Rat{m}$ with $\widehat  
T(x,y)=(T(x),T(y))$, 
consider an optimal coupling $\pi$ for which the infimum in the definition of the distance 
$W_p(\mu,\nu)$ is attained, and define $\widehat{\pi}:=\widehat T_{\#}\pi=\pi\circ \widehat 
T^{-1}$. Then, it follows that $
\widehat{\pi}(A\times\Rat{m})=(\pi\circ\widehat 
T^{-1})(A\times\Rat{m})=\pi(T^{-1}(A)\times 
T^{-1}(\Rat{m}))=\mu(T^{-1}(A))=T_{\#}\mu(A)$.
Hence, $T_{\#}\mu$ is a marginal of $\widehat{\pi}$ and similarly $T_{\#}\nu$, 
i.e., $\widehat{\pi}$ is a coupling between $T_{\#}\mu$ and $T_{\#}\nu$.  
Let $\phi:\Rat{m}\times\Rat{m}\to\Rat{}$ with $\phi(x,y)=\|x-y\|^p$ and $\widehat T$ 
as given above. Then, we obtain from the change of variables formula and the Lipschitz hypothesis that 
\begin{align*}
({\rm LHS})
& =\int_{\Rat{m}\times\Rat{m}}\|\widehat x-\widehat 
y\|^p\widehat{\pi}(d\widehat x,d\widehat y)= 
\int_{\Rat{m}\times\Rat{m}}\phi(\widehat x,\widehat y)\widehat{\pi}(d\widehat
x,d\widehat y) \\
& = \int_{\Rat{n}\times\Rat{n}}\phi\circ\widehat 
T(x,y)\pi(dx,dy)=\int_{\Rat{n}\times\Rat{n}}\phi(T(x),T(y))\pi(dx,dy) \\  
& = \int_{\Rat{n}\times\Rat{n}}\|T(x)-T(y)\|^p\pi(dx,dy)\le 
\int_{\Rat{n}\times\Rat{n}}L^p\|x-y\|^p\pi(dx,dy)=({\rm RHS}).
\end{align*} 
Thus, we get $W_p^p(T_{\#}\mu,T_{\#}\nu)\le ({\rm LHS})\le ({\rm RHS})=L^p W_p^p(\mu,\nu)$, implying the result.
\end{proof}

\begin{proof}[Proof of \cref{lemma:CDF:envelope}]
We show that $\supscr{\F}{up}_\rho[F]$ is continuous and increasing, and hence, it is 
also a CDF, as it takes values in $[0,1]$ (the proof for 
$\supscr{\F}{low}_\rho[F]$ is analogous). Notice first that due to (GI1), i.e., that $F(t)<y\Rightarrow t<F^{-1}(y)$, the mapping $z\mapsto \int_{F(t)}^z(F^{-1}(y)-t)dy$ is strictly increasing for $z\in [F(t),1]$. Combining this fact with continuity of $z\mapsto \int_{F(t)}^z(F^{-1}(y)-t)dy$, we deduce existence of a unique $z\in [F(t),1]$ so that $\supscr{\F}{up}_\rho[F](t)=z$ and $\int_{F(t)}^z(F^{-1}(y)-t)dy=\rho$ for all $t\in [a,\supscr{t}{up}_\rho[F])$. To show that $\supscr{\F}{up}_\rho[F]$ is increasing, let $a\le t_1<t_2<\supscr{t}{up}_\rho[F]$ with $\supscr{\F}{up}_\rho[F](t_1)=z_1$ and $\supscr{\F}{up}_\rho[F](t_2)=z_2$ and assume w.l.o.g. that $F(t_2)<z_1$. Then, we have that 
\begin{align*}
\rho=\int_{F(t_1)}^{z_1}(F^{-1}(y)-t_1)dy\ge\int_{F(t_2)}^{z_1}(F^{-1}(y)-t_1)dy
>\int_{F(t_2)}^{z_1}(F^{-1}(y)-t_2)dy,
\end{align*}
where we exploited that $F$ is increasing in the first inequality. Thus,  we 
get that $z_2>z_1$, because also $\int_{F(t_2)}^{z_2}(F^{-1}(y)-t_2)dy=\rho$. 
To prove continuity, let $t_\nu\to t\in [a,\supscr{t}{up}_\rho[F])$ and 
$\{z_\nu\}_{\nu\in\bN}$ with $\supscr{\F}{up}_\rho[F](t_\nu)=z_\nu$. Then, we 
have that
\begin{align*}
\int_{F(t_\nu)}^{z_\nu}(F^{-1}(y)-t_\nu)dy
= &\int_{F(t)}^{z_\nu}(F^{-1}(y)-t)dy \\
& +\int_{F(t_\nu)}^{F(t)}(F^{-1}(y)-t)dy
+\int_{F(t_\nu)}^{z_\nu}(t-t_\nu)dy,
\end{align*}  
or equivalently, $\int_{F(t)}^{z_\nu}(F^{-1}(y)-t)dy= \rho
-\int_{F(t_\nu)}^{F(t)}(F^{-1}(y)-t)dy
-\int_{F(t_\nu)}^{z_\nu}(t-t_\nu)dy$.
Since $0\le F(t_\nu)<z_\nu\le 1$, and $t_\nu\to t$ we get that 
$\int_{F(t_\nu)}^{z_\nu}(t-t_\nu)dy\to 0$. For the other term, we have w.l.o.g. 
that $F(t_\nu)\le y\le F(t)$. It then follows from (GI2) that $t_\nu\le F^{-1}(y)\le t$ and therefore 
$
\big|\int_{F(t_\nu)}^{F(t)}(F^{-1}(y)-t)dy\big|
\le\int_{F(t_\nu)}^{F(t)}|t_\nu-t|dy \to 0$.
Thus,
\begin{align} \label{envelope:continuity:argument}
\int_{F(t)}^{z_\nu}(F^{-1}(y)-t)dy\to \rho=\int_{F(t)}^z(F^{-1}(y)-t)dy
\end{align}
for a unique $z\in [F(t),1]$. Since $z'\mapsto \int_{F(t)}^{z'}(F^{-1}(y)-t)dy$ 
is strictly increasing (near $z$) and continuous, its inverse is well defined and continuous (see e.g., \cite[Theorem 5, Page 168]{VAZ:00}). Thus, we get from \eqref{envelope:continuity:argument} that  $z_\nu\to z$, establishing continuity of~$\supscr{\F}{up}_\rho[F]$.  

Next, let $F'\in\CD([a,b])$ with $W_1(F,F')\le\rho$. Equivalently, $\int_a^b|F'(t)-F(t)|dt\le\rho$. We show \eqref{CDF:sandwich} by contradiction. Assume  w.l.o.g. that the upper bound in \eqref{CDF:sandwich} is violated, and there exists $t^*$ with  $F'(t^*)>\supscr{\F}{up}_\rho[F](t^*)$. Then necessarily $t^*\in[a,\supscr{t}{up}_\rho[F])$, and since $F'(t^*)>F(t^*)$, (GI1) implies that $F^{-1}(F'(t^*))>t^*$. Hence, $[t^*,F^{-1}(F'(t^*)))$ is nonempty and we get from (GI3) that $F'(t)\ge F(t)$ for all $t\in[t^*,F^{-1}(F'(t^*)))$.  Consequently, we obtain  
\begin{align*}  
\rho & \ge \int_a^b|F'(t)-F(t)|dt  \ge\int_{t^*}^{F^{-1}(F'(t^*))}|F'(t)-F(t)|dt \\
& =\int_{t^*}^{F^{-1}(F'(t^*))}(F'(t)-F(t))dt \ge \int_{t^*}^{F^{-1}(F'(t^*))}(F'(t^*)-F(t))dt \\
& =\int_{F(t^*)}^{F'(t^*)}(F^{-1}(y)-t^*)dy 
>\int_{F(t^*)}^{\supscr{\F}{up}_\rho[F](t^*)}(F^{-1}(y)-t^*)dy=\rho, 
\end{align*}
which is a contradiction. 
\end{proof}

\begin{proof}[Proof of \cref{prop:envelopeconstruction}]
We break the proof into several steps. 

\textit{Step 1: all indices $j_k$ and $i_k$ are well defined and satisfy~\eqref{index:ordering}}. We need to establish that the $\min$ and $\max$ operations for the definitions of these indices are not taken over the empty set. To show this for all $k\in[1:k_{\max}]$, we verify the following Induction Hypothesis (IH):
\begin{align*}
\textup{(IH) \quad For each}\; k\in[1:k_{\max}],\; j_k,\;i_k\;\textup{are well defined},\;j_k<i_k,\;{\rm and}\;b_{i_k,j_k}\ge\rho. 
\end{align*}
All properties of (IH) can be directly checked for $k=1$ by the definition of $j_1$ and $i_1$, and the assumption $b_{N,0}>\rho$. For the general case, let $k\le k_{\max}-1$ and assume that (IH) is fulfilled. Then, $j_{k+1}$ is well defined because $b_{i_k,j_k}\ge\rho$ by (IH). To show this also for $i_{k+1}$ we first establish that $i_k<N$. Indeed, assume on the contrary that $i_k=N$. Then, from the definition of $j_{k+1}$ we have that $b_{i_k,j_{k+1}}<\rho$ and we get from the definition of $k_{\max}$ that $k\ge k_{\max}$, which is a contradiction. Since $i_k<N$, $[i_{k+1}:N]$ is nonempty. Combining this with the fact that $b_{N,j_{k+1}}>\rho$, which follows from the definition of $k_{\max}$ and our assumption $k<k_{\max}$, we deduce that the minimum in the definition of $i_{k+1}$ is taken over a non-empty set. Hence, $i_{k+1}$ is well defined. In addition, we get from the definitions of $j_{k+1}$ and $i_{k+1}$ that $j_{k+1}<i_{k+1}$ and from the definition of $i_{k+1}$ that $b_{i_{k+1},j_{k+1}}\ge\rho$. Thus, we have shown (IH). Finally,  $j_{k_{\max}+1}$ is also well defined because  $b_{i_{k_{\max}},j_{k_{\max}}}\ge\rho$ by (IH). Having established that $j_k$ and $i_k$ are well defined for all $k\in[1:k_{\max}+1]$, \eqref{index:ordering} follows directly from their expressions. 

\textit{Step 2: establishing~\eqref{all:times:ordering}}. By the definition of $j_{k+1}$, we get 
\begin{align}\label{bvalues:property1}
b_{i_k,j_{k+1}} <\rho\quad\forall k\in[1:k_{\max}]. 
\end{align}
In addition, we have that 
\begin{align}  \label{bvalues:property2}
b_{i_{k+1}-1,j_{k+1}} & <\rho\quad\forall k\in[0:k_{\max}].
\end{align}
For $k=0$ this follows from the definition of $j_1$ and $i_1$. To show it also for $k\in[1:k_{\max}]$ we consider two cases. If $b_{i_k+1,j_{k+1}}\ge\rho$, then, by definition, $i_{k+1}=i_k+1$ and we get from \eqref{bvalues:property1} that $b_{i_{k+1}-1,j_{k+1}}=b_{i_k,j_{k+1}}<\rho$. In the other case where $b_{i_k+1,j_{k+1}}<\rho$, \eqref{bvalues:property2} follows directly from the  definition of $i_{k+1}$. Next, note that due to \eqref{index:ordering} and the fact that $i_{k_{\max}+1}=N+1$, the times $\tau_\ell$ are indeed defined for all $\ell\in[i_1:N]$. In addition, for each $k\in[1:k_{\max}]$ we get from \eqref{bvalues:property2} that $\rho-b_{\ell,j_{k+1}}>0$ for all $\ell\in[i_k:i_{k+1}-1]$. Hence, $\Delta t_\ell$ is positive and strictly decreasing with $\ell\in[i_k:i_{k+1}-1]$ and we have from the definition of the $\tau_\ell$'s that 
\begin{align} 
\tau_\ell & <\tau_{\ell'}\quad \forall k\in[1:k_{\max}],\ell,\ell'\in[i_k:i_{k+1}-1]\;{\rm with}\; \ell<\ell'  \label{tauis:ordering:prop1} \\
\tau_{i_{k+1}-1} & <t_{j_{k+1}}\quad\forall k\in[1:k_{\max}]. \label{tauis:ordering:prop2}
\end{align}
By the definition of $j_{k+1}$ we further obtain that 
\begin{align}  \label{bvalues:property3}
b_{i_k,j_{k+1}-1}\ge\rho\quad\forall k\in[1:k_{\max}]. 
\end{align}
From the latter and the definition of $\Delta t_{i_k}$, which implies that  $\Delta t_{i_k}\sum_{l=j_{k+1}}^{i_k}c_l+b_{i_k,j_{k+1}}=\rho$, we get that $b_{i_k,j_{k+1}-1}\ge \Delta t_{i_k}\sum_{l=j_{k+1}}^{i_k} c_l+b_{i_k,j_{k+1}}$, or equivalently, that 
\begin{align*}
& \sum_{l=j_{k+1}-1}^{i_k}(t_l-t_{j_{k+1}-1})c_l-\sum_{l=j_{k+1}}^{i_k}(t_l-t_{j_{k+1}})c_l \ge \Delta t_{i_k}\sum_{l=j_{k+1}}^{i_k} c_l \Leftrightarrow \\ 
& \sum_{l=j_{k+1}}^{i_k}(t_{j_{k+1}}-t_{j_{k+1}-1})c_l \ge \Delta t_{i_k}\sum_{l=j_{k+1}}^{i_k} c_l\Leftrightarrow t_{j_{k+1}}-t_{j_{k+1}-1}\ge \Delta t_{i_k}.
\end{align*}
Thus, we deduce from the definition of $\tau_\ell$ with $\ell\equiv i_k$ that $\tau_{i_k}\ge t_{j_{k+1}-1}$ for  $k\in[1:k_{\max}]$. Using this, and recalling that $\{t_\ell\}_{\ell=0}^N$ are strictly increasing, we get from \eqref{index:ordering}, \eqref{tauis:ordering:prop1}, and \eqref{tauis:ordering:prop2}, that $\{\tau_\ell\}_{\ell=j_1}^N$ are strictly increasing and \eqref{all:times:ordering} is satisfied.   

\textit{Step 3: verification of the formula for $\supscr{\widehat F}{up}$ for $t\in(-\infty,a)\cup[\tau_N,\infty)$}. For $t\in(-\infty,a)$, it follows directly from the definition of the upper CDF envelope. To establish it also when $t\in[\tau_N,\infty)$, it suffices again from the definition of the upper CDF envelope to show that $\tau_N=\supscr{t}{up}_\rho[\widehat F]$, with $\supscr{t}{up}_\rho$ given in the statement of \cref{lemma:CDF:envelope}. To show this, note that since by \eqref{all:times:ordering} $t_{j_{k_{\max}+1}-1}\le\tau_N<t_{j_{k_{\max}+1}}$, we have  
\begin{align*}
\int_{\tau_N}^b(1-\widehat F(t))dt & =\int_{\tau_N}^{t_N}(1-\widehat F(t))dt=\int_{\tau_N}^{t_{j_{k_{\max}+1}}}(1-\widehat F(t))dt \\
&\phantom{=}+\int_{t_{j_{k_{\max}+1}}}^{t_N}(1-\widehat F(t))dt = (t_{j_{k_{\max}+1}}-\tau_N)\sum_{l=j_{k_{\max}+1}}^Nc_l+b_{N,j_{k_{\max}+1}} ,
\end{align*}
which, in turn, equals $\Delta t_N\sum_{l=j_{k_{\max}+1}}^Nc_l+b_{N,j_{k_{\max}+1}}$.
Thus, we get from the definition of $\Delta t_N$ that 
$
\int_{\tau_N}^b(1-\widehat F(t))dt=\frac{\rho-b_{N,j_{k_{\max}+1}}}{\sum_{l=j_{k_{\max}+1}}^Nc_l}\sum_{l=j_{k_{\max}+1}}^Nc_l+b_{N,j_{k_{\max}+1}}=\rho$,
and hence
$\tau_N=\sup\{\tau\in[a,b]\,|\, \int_\tau^b(1-\widehat F(t))dt\ge\rho\}=\supscr{t}{up}_\rho[\widehat F]$.
It remains to verify the formula for $\supscr{\widehat F}{up}$ for all intermediate intervals, which are of the form $[\subscr{t}{beg},\subscr{t}{end})$. To each of these intervals we also associate a right time-instant $\subscr{t}{rt}$. For each $k\in[1:k_{\max}]$, $\subscr{t}{beg}$, $\subscr{t}{end}$, and $\subscr{t}{rt}$ are given by one of the following  cases. 

\textbf{Case 1)} $\subscr{t}{beg}=t_\ell$ and $\subscr{t}{end}=t_{\ell+1}$ with 
$\ell\in[j_k:j_{k+1}-2]$, and $\subscr{t}{rt}=t_{i_k}$; 

\textbf{Case 2)} $\subscr{t}{beg}=t_{j_{k+1}-1}$, 
$\subscr{t}{end}=\tau_{i_k}$, and $\subscr{t}{rt}=t_{i_k}$;

\textbf{Case 3)} $\subscr{t}{beg}=\tau_\ell$ and 
$\subscr{t}{end}=\tau_{\ell+1}$ with $\ell\in[i_k:i_{k+1}-2]$, and $\subscr{t}{rt}=t_{\ell+1}$; 

\textbf{Case 4)} $\subscr{t}{beg}=\tau_{i_{k+1}-1}$, $\subscr{t}{end}=t_{j_{k+1}}$, and $\subscr{t}{rt}=t_{i_{k+1}}$. 

\noindent One can readily check from the formula for $\supscr{\widehat F}{up}$ that 
these cases cover all intermediate intervals. To verify the formula for all $[\subscr{t}{beg},\subscr{t}{end})$ we will exploit the following fact:

\textbf{Fact I)} For each of the Cases~1)--4) and pair $(t,y)$ with $t\in(\subscr{t}{beg},\subscr{t}{end})$ and $y=\supscr{\widehat F}{up}(t)$, it holds that $\widehat F^{-1}(y)=\subscr{t}{rt}$.

\textit{Step 4: Proof of Fact I}. Recall that 
\begin{align} \label{Fup:reminder}
\supscr{\widehat F}{up}(t)=\sup\bigg\{z\in[\widehat F(t),1]\,\Big|\,\int_{\widehat F(t)}^z(\widehat F^{-1}(y)-t)dy\le\rho\bigg\}
\end{align}
and note that
\begin{align} 
\int_{\widehat F(t_j)}^{\widehat F(t_i)}(\widehat F^{-1}(y)-t_j)dy=b_{i,j}\quad\forall\; 0\le j\le i\le N \label{integral:equal:bij}.
\end{align}
%
We first consider Case~1). Let $t\in(t_\ell,t_{\ell+1})$ with $\ell\in[j_k:j_{k+1}-2]$. Then, we have from \eqref{index:ordering} and \eqref{integral:equal:bij} that  
\begin{align*}
\int_{\widehat F(t)}^{\widehat F(t_{i_k})}(\widehat F^{-1}(y)-t)dy & \ge\int_{\widehat F(t_{j_{k+1}-1})}^{\widehat F(t_{i_k})}(\widehat F^{-1}(y)-t_{j_{k+1}-1})dy=b_{i_k,j_{k+1}-1}\ge\rho,
\\
\int_{\widehat F(t)}^{\widehat F(t_{i_k-1})}(\widehat F^{-1}(y)-t)dy &\le\int_{\widehat F(t_{j_k})}^{\widehat F(t_{i_k-1})}(\widehat F^{-1}(y)-t_{j_k})dy=b_{i_k-1,j_k}<\rho,
\end{align*}
where we exploited \eqref{bvalues:property3} and \eqref{bvalues:property2} for each last inequality, respectively. Thus, it follows from \eqref{Fup:reminder} that $\widehat F(t_{i_k-1})<\supscr{\widehat F}{up}(t)\le \widehat F(t_{i_k})$, which implies by (GI4) that $\widehat F^{-1}(\supscr{\widehat F}{up}(t))=t_{i_k}\equiv\subscr{t}{rt}$. For Case~2), let $t\in(t_{j_{k+1}-1},\tau_{i_k})$. Then, we get from \eqref{integral:equal:bij} and the definition of $\tau_{i_k}$ that  
\begin{align*}
\int_{\widehat F(t)}^{\widehat F(t_{i_k})}(\widehat F^{-1}(y)-t)dy & \ge \int_{\widehat F(t_{j_{k+1}-1})}^{\widehat F(t_{i_k})} \! (\widehat F^{-1}(y)-\tau_{i_k})dy 
 =\int_{\widehat F(t_{j_{k+1}-1})}^{\widehat F(t_{i_k})} \hspace*{-1ex} (\widehat F^{-1}(y)-t_{j_{k+1}})dy
 \\
 & +\int_{\widehat F(t_{j_{k+1}-1})}^{\widehat F(t_{i_k})}(t_{j_{k+1}}-\tau_{i_k})dy  =b_{i_k,j_{k+1}}+\Delta t_{i_k}\sum_{l=j_{k+1}}^{i_k}c_l=\rho,
\end{align*}
whereas by arguing precisely as in Case~1), we get that $\int_{\widehat F(t)}^{\widehat F(t_{i_k-1})}(\widehat F^{-1}(y)-t)dy<\rho$. Thus, we deduce  $\widehat F(t_{i_k-1})<\supscr{\widehat F}{up}(t)\le \widehat F(t_{i_k})$, and hence, by (GI4), $\widehat F^{-1}(\supscr{\widehat F}{up}(t))=t_{i_k}\equiv\subscr{t}{rt}$. 
The proof of Fact~I for Cases 3) and 4) follows similar arguments and exploits the orderings \eqref{index:ordering} and \eqref{all:times:ordering}, and we omit it for space reasons.

\textit{Step 5: verification of the formula for $\supscr{\widehat F}{up}$ for $t\in[a,\tau_N)$}. Let any interval $(\subscr{t}{beg},\subscr{t}{end})$ as given by Cases 1)--4), let $t\in(\subscr{t}{beg},\subscr{t}{end})$, $\{t_\nu\}_{\nu\in\bN}\subset(\subscr{t}{beg},\subscr{t}{end})$ with $t_\nu\searrow\subscr{t}{beg}$, and denote $y\equiv\supscr{\widehat F}{up}(t)$, $y_\nu\equiv\supscr{\widehat F}{up}(t_\nu)$, $\nu\in\bN$. Due to Fact I, 
$\widehat F^{-1}(y)=\subscr{t}{rt},\quad \widehat F^{-1}(y_\nu)=\subscr{t}{rt}$ for all $\nu\in\bN$.
We use this together with $
z=\supscr{\widehat F}{up}(t)\Leftrightarrow\int_t^{\widehat F^{-1}(z)}(z-\widehat F(s))ds=\rho
$ and the continuity of $\supscr{\widehat F}{up}$ (which implies $y_\nu\to\subscr{y}{beg}\equiv\supscr{\widehat F}{up}(\subscr{t}{beg})$) to~get 
\begin{align*}
\int_t^{F^{-1}(y)}(y-\widehat F(s))ds & =\int_{t_\nu}^{F^{-1}(y_\nu)}(y_\nu-\widehat F(s))ds \quad\forall \nu\in\bN \Leftrightarrow \\
\int_t^{\subscr{t}{rt}}(y-\widehat F(s))ds & =\int_{t_\nu}^{\subscr{t}{rt}}(y_\nu-\widehat F(s))ds  \quad\forall \nu\in\bN \Leftrightarrow \\
\int_t^{\subscr{t}{rt}}(y-\widehat F(s))ds & =\int_{\subscr{t}{beg}}^{\subscr{t}{rt}}(\subscr{y}{beg}-\widehat F(s))ds \Leftrightarrow \\
\int_t^{\subscr{t}{rt}} \hspace*{-1ex} (y-\subscr{y}{beg})ds+\int_t^{\subscr{t}{rt}} \hspace*{-1ex} (\subscr{y}{beg}-\widehat F(s))ds  & = \hspace*{-1ex} \int_{\subscr{t}{beg}}^t \hspace*{-1ex} (\subscr{y}{beg}-\subscr{y}{low})ds
+\int_t^{\subscr{t}{rt}} \hspace*{-1ex} (\subscr{y}{beg}-\widehat F(s))ds  \Leftrightarrow
\\
(y-\subscr{y}{beg})(\subscr{t}{rt}-t) & =(\subscr{y}{beg}-\subscr{y}{low})(t-\subscr{t}{beg}), 
\end{align*}
with $\subscr{y}{low}=\widehat F(\subscr{t}{beg})$, cf.~\cref{fig:envelope:empirical}. Hence,
$y =\subscr{y}{beg}+(\subscr{y}{beg}-\subscr{y}{low}) 
\frac{t-\subscr{t}{beg}}{\subscr{t}{rt}-t} 
= \subscr{y}{low}+(\subscr{y}{beg}-\subscr{y}{low}) 
\frac{\subscr{t}{rt}-\subscr{t}{beg}}{\subscr{t}{rt}-t}$.
The proof is completed by verifying the formula for $\supscr{\widehat F}{up}$ at $\subscr{t}{beg}$ for each interval given by Cases 1)--4), which follows from the definitions of $y_\ell$ and~$z_\ell$.
\end{proof}

\begin{proof}[Proof of \cref{lemma:reflected:CDF:envelope}]
We exploit the following equivalences for any $F\in\CD([a,b])$ and pair  $(t,y)$ in the graph of its lower and upper CDF envelopes:
\begin{subequations}
\begin{align}
& y=\supscr{\F}{low}_\rho[F](t) \Leftrightarrow \int_{F^{-1}(y)}^t(F(s)-y)ds=\rho \label{lower:env:equiv} \\
& y=\supscr{\F}{up}_\rho[F](t) \Leftrightarrow \int_t^{F^{-1}(y)}(y-F(s))ds=\rho. \label{upper:env:equiv}
\end{align}
\end{subequations}
We also use the following elementary results about the left inverse of a CDF $F\in\CD(\mathbb R)$, defined by $\subscr{F}{left}^{-1}(y):=\inf\{t\in\mathbb R\,|\,F(t)\ge y\}$.

\noindent \textbf{Fact II)} For any $y\in(0,1)$, $F^{-1}(1-y)=a+b-\subscr{\widetilde F}{left}^{-1}(y)$, where  $\widetilde F\equiv \supscr{\F}{refl}_{(\frac{a+b}{2},\frac{1}{2})}[F]$.

\noindent \textbf{Fact III)} For any $y\in[0,1]$ and $t\in\mathbb R$, $\int_t^{\subscr{F}{left}^{-1}(y)}(y-F(s))ds=\int_t^{F^{-1}(y)}(y-F(s))ds$. 

\noindent Next, let $F\in\CD([a,b])$ and  denote $\widetilde F\equiv \supscr{\F}{refl}_{(\frac{a+b}{2},\frac{1}{2})}[F]$ and $\supscr{\widetilde F}{up}\equiv\supscr{\F}{up}_\rho[\widetilde F]$. To prove the result, we show that $\supscr{\F}{low}_\rho[F](t)=\supscr{\F}{refl}_{(\frac{a+b}{2},\frac{1}{2})}\big[\supscr{\widetilde F}{up}](t)$ for any $t$ for which these values are in $(0,1)$. Let $y=1-\supscr{\widetilde F}{up}(a+b-t)=\supscr{\F}{refl}_{(\frac{a+b}{2},\frac{1}{2})}\big[\supscr{\widetilde F}{up}](t)\in(0,1)$. We show that $\int_{F^{-1}(y)}^t(F(s)-y)ds=\rho$, which by \eqref{lower:env:equiv} implies that $\supscr{\F}{low}_\rho[F](t)=y$. Indeed, 
\begin{align*}
\int_{F^{-1}(y)}^t(F(s) &-y)ds =\int_{F^{-1}(1-\supscr{\widetilde F}{up}(a+b-t))}^t(F(s)-(1-\supscr{\widetilde F}{up}(a+b-t)))ds \\
& = \int_{a+b-\subscr{\widetilde F}{left}^{-1}(\supscr{\widetilde F}{up}(a+b-t))}^t(F(s)-(1-\supscr{\widetilde F}{up}(a+b-t)))ds \\
& = \int_{a+b-t}^{\subscr{\widetilde F}{left}^{-1}(\supscr{\widetilde F}{up}(a+b-t))}(\supscr{\widetilde F}{up}(a+b-t)-\widetilde F(s))ds \\
& = \int_{a+b-t}^{\widetilde F^{-1}(\supscr{\widetilde F}{up}(a+b-t))}(\supscr{\widetilde F}{up}(a+b-t)-\widetilde F(s))ds=\rho,
\end{align*}
where we used Fact II in the second equality, that the reflection around $(\frac{a+b}{2},\frac{1}{2})$, i.e., the change of variables $(t,y)\mapsto(a+b-t,1-y)$ is an isometry in the third equality, Fact III in the fourth equality, and the equivalent characterization \eqref{upper:env:equiv} for $y=\supscr{\F}{up}_\rho[F](t)$ in the last equality.
%
\end{proof}

\begin{proof}[Proof of Fact II]
Let $y\in(0,1)$. Then 
\begin{align*}
F^{-1}(1-y) & =\inf\{t\in\mathbb R\,|\,F(t)>1-y\}=\inf F^{-1}((1-y,\infty)) \\
&=\sup F^{-1}((-\infty,1-y])=\sup\{t\in\mathbb R\,|\,F(t)\le 1-y\} \\
& =\sup\{t\in\mathbb R\,|\,1-\widetilde F(a+b-t)\le 1-y\} \\
& =\sup\{a+b-\tau,\tau\in\mathbb R\,|\,1-\widetilde F(\tau)\le 1-y\} \\
& =a+b+\sup\{-\tau,\tau\in\mathbb R\,|\,\widetilde F(\tau)\ge y\} \\
& =a+b-\inf\{\tau\in\mathbb R\,|\,\widetilde F(\tau)\ge y\} =a+b-\subscr{\widetilde F}{left}^{-1}(y), 
\end{align*}
where we used $F$ is increasing and $\inf I=\sup I^c$ for any intervals $I$, $I^c$ with $I\cup I^c=\mathbb R$ in the third equality.  
\end{proof}

\begin{proof}[Proof of Fact III]
To show the result we will prove that $\int_{\subscr{F}{left}^{-1}(y)}^{F^{-1}(y)}(y-F(s))ds=0$. Since $F^{-1}(y)\ge\subscr{F}{left}^{-1}(y)$, it suffices to consider the case of strict inequality. Then, the result follows directly from the fact that $F(s)=y$ for any $s\in(\subscr{F}{left}^{-1}(y),F^{-1}(y))$, which can be readily checked by the definitions of $F^{-1}$ and $\subscr{F}{left}^{-1}$.
\end{proof}

\section{Derivation of the CDF equation}\label{app:MD}
An equation for the Cumulative Distribution Function of $u(\mathbf x,t)$, solution of \eqref{eq:tran}, obeying \cref{assumption:exact:dynamics} and \cref{assumption:unif:existence}, is obtained via the Method of Distributions in three steps. 
First, we rely on the following inequalities for the newly introduced random variable $\Pi(\widetilde {\mathbf x},t)$
\begin{align}\label{eq:pieq}
\frac{\partial \Pi}{\partial t} = -\frac{\partial \Pi}{\partial U} \frac{\partial u}{\partial t}, \quad \nabla \Pi = - \frac{\partial \Pi}{\partial U} \nabla u.
\end{align} 
Second, we multiply \eqref{eq:tran} by $-\frac{\partial \Pi}{\partial U}$ and, accounting for \eqref{eq:pieq}, we obtain a stochastic PDE for $\Pi(U,\mathbf x,t)$:
\begin{align}\label{eq:pi}
\frac{\partial \Pi}{\partial t} + \dot {\mathbf q} (U) \cdot \nabla \Pi = -\frac{\partial \Pi}{\partial U} r(U), \quad \mathbf x \in \Omega, U \in \mathbb R, t>0,
\end{align} 
with $\dot{\mathbf q} = \partial \mathbf q / \partial U$. This formulation is exact in case of smooth solutions of \eqref{eq:tran}~\cite{perthame2002kinetic} and whenever $\nabla \cdot \mathbf q(U)  = 0 $.~\eqref{eq:pi} is defined in an augmented $(d+1)$-dimensional space $\widetilde \Omega = \Omega \times \mathbb R$, and it is subject to initial and boundary conditions that follow from the initial and boundary conditions of the original model
\begin{align*}
    \Pi(U,\mathbf x,t=0) = \Pi_0 = \mathcal H(U-u_0(\mathbf x)), & \quad \widetilde{\mathbf x} \in \widetilde \Omega \\
    \Pi(U,\mathbf x,t) = \Pi_b(U,\mathbf x,t) = \mathcal H(U-u_b(t)), & \quad \mathbf x \in \Gamma, U \in \Omega_U, t>0.
\end{align*}

Finally, since the ensemble average of $\Pi$ is the CDF of $u$, $F_{u(\mathbf x,t)} = \langle \Pi(U,\mathbf x,t) \rangle$, ensemble averaging of \eqref{eq:pi} yields \eqref{eq:MD}.
This equation is subject to initial and boundary conditions 
along $ 
(\Gamma\times\mathbb R )$
\begin{align}
    & F_{u(\mathbf x,t)} = F_{u_0(\mathbf x)}, \quad \widetilde{\mathbf x} \in \widetilde \Omega, t=0 \notag \\
    & F_{u(\mathbf x,t)} = F_{u_b(\mathbf x,t)}, \quad \mathbf x \in \Gamma, 
    U \in \mathbb R, t>0 .
      \label{CDF:PDE}
\end{align}
%
The relaxation of Assumptions~\ref{assumption:exact:dynamics} and~\ref{assumption:unif:existence} leads to different (and often approximated) CDF equations: we refer to~\cite{boso-2014-cumulative,boso-2019-hyperbolic} for a complete discussion.

\vspace*{-1ex}
\bibliographystyle{siamplain}
\bibliography{alias,references}

\begin{thebibliography}{10}

\bibitem{AB-DD-AD-BM-GR:13}
{\sc A.~Ben-Tal, D.~D. Hertog, A.~D. Waegenaere, B.~Melenberg, and G.~Rennen},
  {\em Robust solutions of optimization problems affected by uncertain
  probabilities}, Manage. Sci., 59 (2013), p.~341–357.

\bibitem{DB-DBB-CC:11}
{\sc D.~Bertsimas, D.~B. Brown, and C.~Caramanis}, {\em Theory and applications
  of robust optimization}, SIAM Rev., 53 (2011), p.~464–501.

\bibitem{JB-YK-KM:19}
{\sc J.~Blanchet, Y.~Kang, and K.~Murthy}, {\em Robust {W}asserstein profile
  inference and applications to machine learning}, J. Appl. Prob., 56 (2019),
  pp.~830--857.

\bibitem{DB-JC-SM:19-tac}
{\sc D.~Boskos, J.~Cort\'es, and S.~Martinez}, {\em Data-driven ambiguity sets
  with probabilistic guarantees for dynamic processes}, IEEE Trans. Aut.
  Contr.,  (2019).
\newblock Submitted. Available at \url{https://arxiv.org/abs/1909.11194}.

\bibitem{DB-JC-SM:19}
{\sc D.~Boskos, J.~Cort\'es, and S.~Mart{\'\i}nez}, {\em Dynamic evolution of
  distributional ambiguity sets and precision tradeoffs in data assimilation},
  in {E}uropean {C}ontrol {C}onference, Naples, Italy, June 2019,
  pp.~2252--2257.

\bibitem{boso-2014-cumulative}
{\sc F.~Boso, S.~V. Broyda, and D.~M. Tartakovsky}, {\em Cumulative
  distribution function solutions of advection-reaction equations with
  uncertain parameters}, Proc. Roy. Soc. A, 470 (2014), p.~20140189.

\bibitem{boso-2019-hyperbolic}
{\sc F.~Boso and D.~M. Tartakovsky}, {\em Data-informed method of distributions
  for hyperbolic conservation laws}, SIAM J. Sci. Comput., 42 (2020),
  pp.~A559--A583.

\bibitem{buckley1942mechanism}
{\sc S.~E. Buckley, M.~Leverett, et~al.}, {\em Mechanism of fluid displacement
  in sands}, Trans. AIME, 146 (1942), pp.~107--116.

\bibitem{AC-JC:20-tac}
{\sc A.~Cherukuri and J.~Cort\'es}, {\em Cooperative data-driven
  distributionally robust optimization}, IEEE Trans. Aut. Contr., 65 (2020).
\newblock To appear.

\bibitem{ED-YY:10}
{\sc E.~Delage and Y.~Ye}, {\em Distributionally robust optimization under
  moment uncertainty with application to data-driven problems}, Operations
  Research, 58 (2010), p.~595–612.

\bibitem{SD-MS-RS:13}
{\sc S.~Dereich, M.~Scheutzow, and R.~Schottstedt}, {\em Constructive
  quantization: {A}pproximation by empirical measures}, Annales de l'Institut
  Henri Poincaré, Probabilités et Statistiques, 49 (2013), p.~1183–1203.

\bibitem{dvoretzky1956asymptotic}
{\sc A.~Dvoretzky, J.~Kiefer, and J.~Wolfowitz}, {\em Asymptotic minimax
  character of the sample distribution function and of the classical
  multinomial estimator}, The Annals of Mathematical Statistics,  (1956),
  pp.~642--669.

\bibitem{esfahani2018data}
{\sc P.~M. Esfahani and D.~Kuhn}, {\em Data-driven distributionally robust
  optimization using the {W}asserstein metric: Performance guarantees and
  tractable reformulations}, Mathematical Programming, 171 (2018),
  pp.~115--166.

\bibitem{NF-AG:15}
{\sc N.~Fournier and A.~Guillin}, {\em On the rate of convergence in
  {W}asserstein distance of the empirical measure}, Probability Theory and
  Related Fields, 162 (2015), p.~707–738.

\bibitem{RG-AJK:16}
{\sc R.~Gao and A.~Kleywegt}, {\em Distributionally robust stochastic
  optimization with {W}asserstein distance}, arXiv preprint arXiv:1604.02199,
  (2016).

\bibitem{YG-KB-ED-ZH-THS:18}
{\sc Y.~Guo, K.~Baker, E.~Dall’Anese, Z.~Hu, and T.~H. Summers}, {\em
  Data-based distributionally robust stochastic optimal power flow—{P}art
  {I}: Methodologies}, IEEE Transactions on Power Systems, 34 (2018),
  pp.~1483--1492.

\bibitem{RJ-MR-GX:19}
{\sc R.~Jiang, M.~Ryu, and G.~Xu}, {\em Data-driven distributionally robust
  appointment scheduling over {W}asserstein balls}, arXiv preprint
  arXiv:1907.03219,  (2019).

\bibitem{ANK:33}
{\sc A.~N. Kolmogorov}, {\em Sulla determinazione emp{\'\i}rica di uma legge di
  distribuzione}, Giornale dell' Istituto Italiano degli Attuari, 4 (1933).

\bibitem{lebacque2005first}
{\sc J.-P. Lebacque}, {\em First-order macroscopic traffic flow models:
  Intersection modeling, network modeling}, in Transportation and Traffic
  Theory. Flow, Dynamics and Human Interaction. 16th International Symposium on
  Transportation and Traffic Theory, University of Maryland, College Park,
  2005.

\bibitem{DL-DF-SM:19-ecc}
{\sc D.~Li, D.~Fooladivanda, and S.~Mart{\'\i}nez}, {\em Data-driven variable
  speed limit design for highways via distributionally robust optimization}, in
  {E}uropean {C}ontrol {C}onference, Napoli, Italy, June 2019, pp.~1055--1061.

\bibitem{massart1990tight}
{\sc P.~Massart}, {\em The tight constant in the
  {D}voretzky-{K}iefer-{W}olfowitz inequality}, The Annals of Probability,
  (1990), pp.~1269--1283.

\bibitem{ABO:95}
{\sc A.~B. Owen}, {\em Nonparametric likelihood confidence bands for a
  distribution function}, Journal of the American Statistical Association, 90
  (1995), pp.~516--521.

\bibitem{perthame2002kinetic}
{\sc B.~Perthame}, {\em Kinetic Formulation of Conservation Laws}, vol.~21,
  Oxford University Press, 2002.

\bibitem{pflug2007ambiguity}
{\sc G.~Pflug and D.~Wozabal}, {\em Ambiguity in portfolio selection},
  Quantitative Finance, 7 (2007), pp.~435--442.

\bibitem{racke1992lectures}
{\sc R.~Racke}, {\em Lectures on Nonlinear Evolution Equations: Initial Value
  Problems}, vol.~19, Springer, 1992.

\bibitem{FS:15}
{\sc F.~Santambrogio}, {\em Optimal Transport for Applied Mathematicians},
  Springer, 2015.

\bibitem{shapiro2017distributionally}
{\sc A.~Shapiro}, {\em Distributionally robust stochastic programming}, SIAM
  Journal on Optimization, 27 (2017), pp.~2258--2275.

\bibitem{shapiro2004class}
{\sc A.~Shapiro and S.~Ahmed}, {\em On a class of minimax stochastic programs},
  SIAM Journal on Optimization, 14 (2004), pp.~1237--1249.

\bibitem{NVS:44}
{\sc N.~V. Smirnov}, {\em Approximate laws of distribution of random variables
  from empirical data}, Uspekhi Matematicheskikh Nauk, 10 (1944), pp.~179--206.

\bibitem{tartakovsky2017method}
{\sc D.~M. Tartakovsky and P.~A. Gremaud}, {\em Method of distributions for
  uncertainty quantification}, Handbook of Uncertainty Quantification,  (2017),
  pp.~763--783.

\bibitem{venturi-2013-exact}
{\sc D.~Venturi, D.~M. Tartakovsky, A.~M. Tartakovsky, and G.~E. Karniadakis},
  {\em Exact pdf equations and closure approximations for advective-reactive
  transport}, Journal of Computational Physics, 243 (2013), pp.~323--343.

\bibitem{CV:03}
{\sc C.~Villani}, {\em Topics in Optimal Transportation}, no.~58, American
  Mathematical Society, 2003.

\bibitem{JW-FB:19}
{\sc J.~Weed and F.~Bach}, {\em Sharp asymptotic and finite-sample rates of
  convergence of empirical measures in {W}asserstein distance}, Bernoulli, 25
  (2019), pp.~2620--2648.

\bibitem{wikle2007bayesian}
{\sc C.~K. Wikle and L.~M. Berliner}, {\em A {B}ayesian tutorial for data
  assimilation}, Physica D: Nonlinear Phenomena, 230 (2007), pp.~1--16.

\bibitem{VAZ:00}
{\sc V.~A. Zorich}, {\em Mathematical Analysis I}, Springer, 2003.

\end{thebibliography}
\end{document}